\documentclass[12pt]{amsart}

\author{Paul \textsc{Poncet}}
\address{CMAP, \'{E}cole Polytechnique, Route de Saclay, 91128 Palaiseau Cedex, France \\
and INRIA, Saclay--\^{I}le-de-France}

\email{poncet@cmap.polytechnique.fr}

\usepackage[colorinlistoftodos,bordercolor=orange,backgroundcolor=orange!20,linecolor=orange,textsize=scriptsize]{todonotes}
\usepackage{stmaryrd}
\usepackage{amssymb}
\usepackage{amsmath}
\usepackage{graphicx}
\usepackage{t1enc}
\usepackage{color}
\usepackage{yhmath}
\usepackage{calrsfs} 
\usepackage{times}
\usepackage{ifpdf} 
\usepackage{bbm}
\newcommand{\ddint}{\int^{\scriptscriptstyle\infty}\!\!}
\newcommand{\dint}[1]{\int^{\scriptscriptstyle\infty}_{#1}\!\!}

\newcommand{\rotm}[1]{\rotatebox[origin=c]{-90}{#1}}
\newcommand{\rotp}[1]{\rotatebox[origin=c]{90}{#1}}

\newtheorem{theorem}{Theorem}[section]
\newtheorem{corollary}[theorem]{Corollary}
\newtheorem{proposition}[theorem]{Proposition}

\theoremstyle{definition}
\newtheorem{definition}[theorem]{Definition}
\newtheorem{example}[theorem]{Example}
\newtheorem{examples}[theorem]{Examples}

\newtheorem{remark}[theorem]{Remark}

\newtheorem{problem}[theorem]{Problem}
\newenvironment{acknowledgements}[1][]{\par\vspace{0.5cm}\noindent\textbf{Acknowledgements#1.} }{\par}

\begin{document}

\title{Representation of maxitive measures: \\an overview}

\date{\today}



\begin{abstract}
Idempotent integration is an analogue of Lebesgue integration where $\sigma$-maxitive measures replace $\sigma$-additive measures. 
In addition to reviewing and unifying several Radon--Niko\-dym like theorems proven in the literature for the idempotent integral, we also prove new results of the same kind. 
\end{abstract}


\maketitle

\section{Introduction}


Maxitive measures were introduced by Shilkret \cite{Shilkret71} as an analogue of classical finitely additive measures or charges with the supremum operation, denoted by $\oplus$, in place of the addition $+$. 
A \textit{maxitive measure} on a $\sigma$-algebra $\mathrsfs{B}$ is then a map $\nu : \mathrsfs{B} \rightarrow \overline{\mathbb{R}}_+$ such that $\nu(\emptyset) = 0$ and 
\[
\nu(B_1 \cup B_2) = \nu(B_1) \oplus \nu(B_2), 
\]
for all $B_1, B_2 \in \mathrsfs{B}$. 
It is \textit{$\sigma$-maxitive} if it commutes with countable unions of elements of $\mathrsfs{B}$. 

In this paper we are interested in representing maxitive measures $\nu$ under the form 
\[
\nu(B) = \dint{B} f \odot d\tau, 
\]
where $\dint{B} f \odot d\tau$ denotes the \textit{idempotent $\odot$-integral} of the measurable map $f$ on $B$ with respect to the maxitive measure $\tau$. 
Here $\odot$ is a \textit{pseudo-multiplication}, i.e.\ an associative binary relation satisfying a series of natural properties. If $\odot$ is the usual multiplication (resp.\ the minimum $\wedge$), then the idempotent $\odot$-integral specializes to the Shilkret integral \cite{Shilkret71} (resp.\ the Sugeno integral \cite{Sugeno74}).  


Idempotent integration has been rediscovered under various forms and studied by several authors with motivations from dimension theory and fractal geometry, optimization, capacities and large deviations of random processes, fuzzy sets and possibility theory, decision theory, idempotent analysis and max-plus (tropical) algebra. 

Because of these numerous fields of application, the wording around maxitive measures is not unique, thus deserves to be reviewed. 
The term of \textit{idempotent integration} that we use was coined by Maslov and derived from the mathematical area of \textit{idempotent analysis} originally developed by 
Kolokoltsov and Maslov \cite{Kolokoltsov89a, Kolokoltsov89b}. 



Many authors have focused on the search for Radon--Nikodym like theorems with respect to the idempotent $\odot$-integral, since the existence of Radon--Nikodym derivatives is often crucial in applications. 
Sugeno and Murofushi \cite{Sugeno87} actually showed that, if $\nu$ and $\tau$ are $\sigma$-maxitive measures on a $\sigma$-algebra $\mathrsfs{B}$, with $\tau$ $\sigma$-$\odot$-finite and $\sigma$-principal,  then $\nu$ is $\odot$-absolutely continuous with respect to $\tau$ if and only if there exists some $\mathrsfs{B}$-measurable map $c : E \rightarrow \overline{\mathbb{R}}_+$ such that 
$
\nu(B) = \dint{B} c \odot d\tau
$ 
for all $B\in \mathrsfs{B}$. 


This result looks like the classical Radon--Nikodym theorem, except that one needs an unusual condition on the dominating measure $\tau$, namely \textit{$\sigma$-principality}. 
This condition roughly says that every $\sigma$-ideal of $\mathrsfs{B}$ has a  greatest element ``modulo negligible sets''. Although $\sigma$-finite  $\sigma$-additive measures are always $\sigma$-principal, this is not true for $\sigma$-finite $\sigma$-maxitive measures. 
Moreover, the conditions of $\sigma$-principality and $\sigma$-$\odot$-finiteness together are essential in the Sugeno--Murofushi theorem: see \cite{Poncet13e} where I showed that a converse statement holds. 

After the article \cite{Sugeno87}, many results of Radon--Nikodym flavour for maxitive measures have been published. This is the case of Agbeko \cite{Agbeko95}, de Cooman \cite{deCooman97}, Akian \cite{Akian99}, Barron, Cardaliaguet, and Jensen \cite{Barron00}, Puhalskii \cite{Puhalskii01}, and Drewnowski \cite{Drewnowski09}. 
By linking several properties of maxitive measures together (see Table~\ref{fig:Schema1Chapitre1}), we shall see why some of these results are already encompassed in the Sugeno--Murofushi theorem. 
In addition, we shall prove a new Radon--Nikodym type theorem in the case where the $\sigma$-maxitive measures $\nu$ and $\tau$ are \textit{associated} (meaning that they are ``strongly dominated'' by a common $\sigma$-maxitive measure). 

\begin{table}[ht]
	\centering
	\begin{tabular}{ccccc}
	\hline
 &  & of bounded variation & $\Longrightarrow$ & finite \\
 &  & \rotm{$\Longrightarrow$} & & \textit{\tiny{if optimal}}\rotp{$\Longrightarrow$} \rotm{$\Longrightarrow$}\textit{\textcolor{white}{\tiny{----------}}} \\
 &  & exhaustive & & \fbox{$\sigma$-finite} \\
 &  & \rotm{$\Longleftrightarrow$} & & \textit{\tiny{if optimal}}\rotp{$\Longrightarrow$} \rotm{$\Longrightarrow$}\textit{\textcolor{white}{\tiny{----------}}} \\
 &  & optimal & & semi-finite \\
 &  & \rotm{$\Longrightarrow$} & & \\
 &  & essential & & \\
 &  & \rotm{$\Longrightarrow$} & & \\
 &  & \fbox{$\sigma$-principal} & $\Longrightarrow$ & autocontinuous \\
 &  & \textit{\textcolor{white}{\tiny{Zorn}}}\rotm{$\Longrightarrow$} \rotp{$\Longrightarrow$}\textit{\tiny{Zorn}} & & \\
 &  & CCC & &  \\
 &  & \rotm{$\Longrightarrow$} & & \\
 &  & localizable & & \\
	\hline
 & & & & 	
	\end{tabular}  
	\caption{Many properties of $\sigma$-maxitive measures defined on a $\sigma$-algebra are considered in this paper; we shall prove many links between these properties, that we have represented here as a summary. The conditions (surrounded in the figure) of $\sigma$-finiteness and $\sigma$-principality taken together are equivalent to the Radon--Nikodym property, as recalled by Theorem~\ref{thm:rncarac}. Note that for $\sigma$-additive measures, $\sigma$-finiteness implies $\sigma$-principality, while this is not the case for $\sigma$-maxitive measures. }
	\label{fig:Schema1Chapitre1}
\end{table}



The paper is organized as follows. 
Section~\ref{secmax} introduces the notion of $\sigma$-maxitive measure and recalls some key theorems and examples. 
Maxitive measures that can be represented as essential suprema are studied in Section~\ref{sec:esssup}; 
we also discuss Barron et al.'s theorem whose proof draws a link between $\sigma$-maxitive measures and classical $\sigma$-additive measures. 
Section~\ref{secshi} develops the idempotent $\odot$-integral and its properties. 
In Section~\ref{secRN} we review existing Radon--Nikodym theorems for the idempotent $\odot$-integral and prove a variant that generalizes results due to de Cooman and Puhalskii; we also make the connection with Section~\ref{sec:esssup}. 
Section~\ref{sec:optim} focuses on the important particular case of optimal measures, i.e.\ maxitive fuzzy measures. 
Section~\ref{sec:poss} proposes a novel definition for possibility measures, relying on the concept of $\sigma$-principality developed in Section~\ref{secRN}. 


\section{Preliminaries on maxitive measures}\label{secmax}

\subsection{Notations} 

Let $E$ be a nonempty set. A \textit{prepaving} on $E$ is a collection of subsets of $E$ containing the empty set and closed under finite unions. 
A collection of subsets of $E$ containing $E$, the empty set, and closed under countable unions and the formation of complements is a \textit{$\sigma$-algebra}. 
When explicitly considering a $\sigma$-algebra, we preferentially denote it by $\mathrsfs{B}$ instead of $\mathrsfs{E}$, and $(E, \mathrsfs{B})$ is referred to as a \textit{measurable space}. 
In a $\sigma$-algebra $\mathrsfs{B}$, a \textit{$\sigma$-ideal} is a nonempty subset $\mathrsfs{I}$ of $\mathrsfs{B}$ that is closed under countable unions and such that $B \subset I \in \mathrsfs{I}$ and $B \in \mathrsfs{B}$ imply $B \in \mathrsfs{I}$. 

Assume in all the sequel that $\mathrsfs{E}$ is a prepaving on $E$. 
We write $\mathbb{R}$ (resp.\ $\mathbb{R}_+)$ for the set of real numbers (resp.\ nonnegative real numbers), and $\overline{\mathbb{R}}_+$ for $\mathbb{R}_+ \cup \{ \infty \}$.  
A \textit{set function} on $\mathrsfs{E}$ is a map $\tau : \mathrsfs{E} \rightarrow \overline{\mathbb{R}}_+$ equal to zero at the empty set. A set function $\tau$ is 
\begin{itemize}
	\item \textit{monotone} if $\tau(G) \leqslant \tau(G')$ for all $G, G' \in \mathrsfs{E}$ such that $G \subset G'$, 
	\item \textit{normed} if $\sup_{G \in \mathrsfs{E}} \tau(G) = 1$, 
	\item \textit{null-additive} if $\tau(G \cup N) = \tau(G)$ for all $G, N \in \mathrsfs{E}$ with $\tau(N) = 0$, 
	\item \textit{finite} if $\tau(G) < \infty$ for every $G \in \mathrsfs{E}$, 
	\item \textit{$\sigma$-finite} if $\tau(G_n) < \infty$ for all $n$, where $(G_n)$ is a countable family of elements of $\mathrsfs{E}$ covering $E$, 
	\item \textit{continuous from below} 
	if $\tau(G) = \lim_n \tau(G_n)$, for all $G_1 \subset G_2 \subset \ldots \in \mathrsfs{E}$ such that $G = \bigcup_n G_n \in \mathrsfs{E}$. 
\end{itemize}

We shall need the following notion of negligibility. If $\tau$ is a null-additive monotone set function on $\mathrsfs{E}$, a subset $N$ of $E$ is $\tau$-\textit{negligible} if it is contained in some $G \in \mathrsfs{E}$ such that $\tau(G) = 0$. A property $P(x)$ ($x\in E$) is satisfied \textit{$\tau$-almost everywhere} (or \textit{$\tau$-a.e.}\ for short) if there exists some negligible subset $N$ of $E$ such that $P(x)$ is true, for all $x \in E \setminus N$. 



\subsection{Definition of maxitive measures}\label{subsec:def}

In this section, $\mathrsfs{E}$ will denote a prepaving on some nonempty set $E$. 

A \textit{maxitive} (resp.\ \textit{completely maxitive}) \textit{measure} on $\mathrsfs{E}$ is a set function $\nu$ on $\mathrsfs{E}$ such that, for every finite (resp.\ arbitrary) family $\{G_j\}_{j\in J}$ of elements of $\mathrsfs{E}$ with $\bigcup_{j \in J} G_j \in \mathrsfs{E}$,  
\begin{equation}\label{eqmax}
\nu(\bigcup_{j\in J} G_j) = \bigoplus_{j \in J} \nu(G_j). 
\end{equation}
A \textit{$\sigma$-maxitive} measure is a maxitive measure which is continuous from below. 
One should note that a $\sigma$-maxitive measure does not necessarily commute with \textit{intersections} of nonincreasing sequences, unlike $\sigma$-additive measures; $\sigma$-maxitive measures with this property were called \textit{optimal measures} by Agbeko \cite{Agbeko95}, see Section~\ref{sec:optim}. 

\begin{remark}
The term ``maxitive'' qualifying a set function that satisfies Equation~\eqref{eqmax} was coined by Shilkret \cite{Shilkret71}, and has been widely used, especially in the fields of probability theory and fuzzy theory. 
However, one can find many other terms in the literature for maxitive or $\sigma$-maxitive measures, or closely related notions, say: 
 \textit{$f$-additive} or \textit{fuzzy additive measures} \cite{Sugeno74, Murofushi93, Wang92}, 
 \textit{contactability measures} \cite{Wang82}, 
 \textit{measures of type $\bigvee$} \cite{Candeloro87}, 
 \textit{idempotent measures} \cite{Maslov87, Akian99}, 
 \textit{max-measures} \cite{Sugeno87},
 \textit{stable measures} \cite{Falconer90},
 \textit{(generalized) possibility measures} \cite{ElRayes94, Mesiar95}, 
 \textit{cost measures} \cite{Akian95, Bernhard00}, 
 \textit{semi-additive measures} \cite{Gelman97},
 \textit{performance measures} \cite{delMoral98}, 
 \textit{sup-decomposable measures} \cite{Mesiar99}, 
 \textit{set-additive measures} \cite{Appell05, Mallet-Paret10a, Mallet-Paret10b}, 
\textit{capacities with the AM property} \cite{Cerda07}. 
 
As for completely maxitive measures, one finds: 
(\textit{generalized}) \textit{possibility measures} \cite{Zadeh78, Shafer87, Dubois88, deCooman97, Wang92}, 
\textit{sup-measures} \cite{Norberg86, OBrien90}, 
\textit{idempotent measures} when $\mathrsfs{E} = 2^E$ or \textit{$\tau$-maxitive measures} for general $\mathrsfs{E}$ \cite{Puhalskii01}, 
\textit{supremum-preserving measures} \cite{Kraetschmer03}. 

Some differences may appear amongst these notions, essentially depending on the choice of the range of the measure and on the structure of the space $(E,\mathrsfs{E})$. 
See also the historical notes in \cite[Appendix~B]{Puhalskii01}. 

The term ``possibility measure'' does not have a unanimous definition: it mainly oscillates between ``normed $\sigma$-maxitive measure'' and ``normed completely maxitive measure'' (and we shall propose in Section~\ref{sec:poss} a different definition). 
Note that \textit{possibility theory} refers to a specific mathematical theory that makes use of the concept of possibility measure in the latter sense and deals with some types of uncertainty and incomplete information. After Zadeh \cite{Zadeh78}, who coined the term and introduced this theory as an extension of fuzzy sets and fuzzy logic, Dubois and Prade must be cited as major contributors in its development; we refer the reader to their monograph \cite{Dubois88} and the recent surveys \cite{Dubois11} and \cite{Dubois12} where several fields of applications of possibility theory are given. 
\end{remark}


Note that every maxitive measure is null-additive and monotone. 
Actually a much stronger property than monotonicity holds, namely the alternating property. 
For a map $f : \mathrsfs{E} \rightarrow \mathbb{R} \cup \{ \pm\infty \}$ we classically define $\mathit{\Delta}_{G_1} \ldots \mathit{\Delta}_{G_n} f(G)$ after Choquet \cite{Choquet54} by iterating the formula $\mathit{\Delta}_{G_1} f(G) = f(G \cup G_1) - f(G)$ (with the convention that $-\infty + \infty = \infty - \infty = 0$). Then $f$ is \textit{alternating of infinite order} (or \textit{alternating} for short) if 
\[
(-1)^{n+1} \mathit{\Delta}_{G_1} \ldots \mathit{\Delta}_{G_n} f(G) \geqslant 0, 
\] 
for all $n \in \mathbb{N}\setminus\{0\}$, $G, G_1, \ldots, G_n \in \mathrsfs{E}$, where $\mathbb{N}$ denotes the set of nonnegative integers. 
Nguyen and Bouchon-Meunier \ \cite{Nguyen03} gave a combinatorial proof of the fact that every finite maxitive measure is alternating (see also  
Harding et al.\ \cite[Theorem~6.2]{Harding97}). This is actually true for every (finite or not) maxitive measure, as the following proposition states. 

\begin{proposition}\label{prop:alt}
Every maxitive measure on $\mathrsfs{E}$ is alternating.
\end{proposition}

\begin{proof}
Recall the convention $\infty - \infty = 0$. We write $s \wedge t$ for the infimum of $\{s, t\}$. 
Let $G_1, \ldots, G_n \in \mathrsfs{E}$, and define $\nu_0(G) = -\nu(G)$, $\nu_{n}(G) = (-1)^{n+1} \mathit{\Delta}_{G_n} \ldots \mathit{\Delta}_{G_1} \nu(G)$. A proof by induction shows that the property ``$\nu_n(G \cup G') = \nu_n(G) \wedge \nu_n(G')$ and $\nu_n(G) = 0 \oplus(\nu_{n-1}(G) - \nu_{n-1}(G_n)) \geqslant 0$, for all $G,G' \in \mathrsfs{E}$'' holds for all $n \in \mathbb{N}\setminus\{0\}$. 
\end{proof}

\subsection{Elementary and advanced examples}

Here we collect some examples given in the literature, especially on metric spaces where maxitive measures appear naturally. Some examples are also linked with extreme value theory, which is the branch of probability theory that aims at the modelling of rare events.

\begin{example}[Essential supremum]\label{ex:esssup}
Let $\tau$ be a null-additive monotone set function, and let $f : E \rightarrow \overline{\mathbb{R}}_+$ be a map. 
We write $\{ f > t \}$ for the subset $\{ x \in E : f(x) > t \}$. 
If one sets 
\[
\nu(G) = \inf \{ t > 0 : G \in \mathrsfs{I}_t \}
\] 
with $\mathrsfs{I}_t := \{ G \in \mathrsfs{E} : G \cap \{ f > t \} \mbox{ is $\tau$-negligible} \}$, 
then $\nu$ 
is a maxitive measure, called the \textit{$\tau$-essential supremum} of $f$, and we write 
\begin{equation}\label{eq:esssup}
\nu(G) = \bigoplus_{x \in G}^{\tau} f(x). 
\end{equation}
In this case, $f$ is a \textit{relative density} of $\nu$ (with respect to $\tau$). Sufficient conditions for the existence of a relative density, when $\nu$ and $\tau$ are given, are discussed in  Section~\ref{sec:esssup}. 
\end{example}

\begin{example}[Cardinal density of a maxitive measure]\label{ex:denscard}
In the previous example, one can take for $\tau$ the maxitive measure $\delta_{\#}$ defined by $\delta_{\#}(G) = 1$ if $G$ is nonempty, $\delta_{\#}(G) = 0$ otherwise. Then the essential supremum in Equation~\eqref{eq:esssup} reduces to an ``exact'' supremum, i.e.\ 
\begin{equation}\label{eqn:den}
\nu(G) = \bigoplus_{x \in G}^{\delta_{\#}} f(x) = \bigoplus_{x \in G} f(x). 
\end{equation}
In this special case we say that $f$ is a \textit{cardinal density} of $\nu$. Note also that a maxitive measure with a cardinal density is necessarily completely maxitive. 
Conversely, complete maxitivity happens to be a sufficient condition for guaranteeing the existence of a cardinal density. I treated this question in detail in \cite{Poncet10} and \cite{Poncet12b}. 
\end{example}

\begin{examples}[Measures of non-compactness]\label{ex:monc}
Let $E$ be a Banach space. 
Following Appell \cite{Appell05}, a \textit{measure of non-compactness} (or \textit{monc} for short) on $E$ is a maxitive measure $\nu$ on the collection of bounded subsets of $E$, satisfying the following axioms, for all bounded subsets $B$ of $E$: 
\begin{itemize}
	\item $\nu(B + K) = \nu(B)$, for all compact subsets $K$ in $E$, 
	\item $\nu(\lambda \cdot B) = \lambda\nu(B)$, for all $\lambda > 0$, 
	\item $\nu(\overline{\operatorname{co}}(B)) = \nu(B)$, where $\overline{\operatorname{co}}(B)$ is the closed convex hull of $B$. 
\end{itemize}
The definition may differ from one author to the other, see e.g.\ Mallet-Paret and Nussbaum \cite{Mallet-Paret10a, Mallet-Paret10b} for a quite different list of axioms. 
Note that if $E = \mathbb{R}^d$, then $\nu(B) = 0$ for all bounded subsets $B$. 
As Appell recalled, three important examples of moncs appear in the literature, namely the \textit{ball monc} (or \textit{Hausdorff monc}) 
\begin{equation*}
	\begin{split}
		\alpha(B) = \inf\{ t > 0: & \mbox{ there are finitely many balls } \\
		& \mbox{ of radius $t$ covering } B \}; 
	\end{split}
\end{equation*}
the \textit{set monc} (or \textit{Kuratowski monc}) 
\begin{equation*}
	\begin{split}
		\beta(B) = \inf\{ t > 0: &  \mbox{ there are finitely many subsets} \\ 
		 & \mbox{ of diameter at most $t$ covering } B \}; 
	\end{split}
\end{equation*}
and the \textit{lattice monc} (or \textit{Istr\u{a}\c{t}escu monc}) 
\begin{equation*}
	\begin{split}
		\gamma(B) = \sup\{ t > 0 : & \mbox{ there is a sequence } (x_n)_n \mbox{ in } B \\
		& \mbox{ with } \left\| x_m - x_n \right\| \geqslant t \mbox{ for } m \neq n \}, 
	\end{split}
\end{equation*}
and we have the classical relations $\alpha \leqslant \gamma \leqslant \beta \leqslant 2.\alpha$. 
Since moncs vanish on compact subsets, hence on singletons, they are a source of examples of maxitive measures with no cardinal density. 
\end{examples}

\begin{examples}[Dimensions]
\textcolor{white}{}
\begin{itemize}
	\item If $E$ is a topological space, the topological dimension is a maxitive measure on the collection of its closed subsets (see e.g.\ Nagata \cite[Theorem~VII-1]{Nagata83}). If $E$ is normal, the topological dimension is even $\sigma$-maxitive \cite[Theorem~VII-2]{Nagata83}. 
	\item If $E$ is a metric space, the Hausdorff dimension and the packing-dimension are $\sigma$-maxitive measures on $2^E$, and the upper box dimension 
	is a maxitive measure on $2^E$ (see e.g.\ Falconer \cite{Falconer90}). 
	\item If $E$ is the Cantor set $\{0, 1\}^{\mathbb{N}}$, the constructive Hausdorff dimension and the constructive packing-dimension are completely maxitive measures on $2^E$, see Lutz \cite{Lutz03, Lutz05}. 
	\item If $E$ is the set of positive integers, the zeta dimension is a maxitive measure on $2^E$, see Doty et al.\ \cite{Doty05}. 
\end{itemize}
\end{examples}

\begin{example}[Random closed sets]
Let $(\mathit{\Omega}, \mathrsfs{A}, P)$ be a probability space and $E$ be a locally-compact, separable, Hausdorff topological space. We denote by $\mathrsfs{F}$ the collection of  closed subsets of $E$, and by $\mathrsfs{K}$ the collection of compact subsets. 
A random closed set is a measurable map $C : \mathit{\Omega} \rightarrow \mathrsfs{F}$. 
For measurability a $\sigma$-algebra on $\mathrsfs{F}$ is needed. The usual $\sigma$-algebra considered is the Borel $\sigma$-algebra generated by the Vietoris (or \textit{hit-and-miss}) topology on $\mathrsfs{F}$. 
Choquet's fundamental theorem is that the distribution of a random closed set $C$ is characterized by its Choquet capacity $T : \mathrsfs{K} \rightarrow [0, 1]$ defined by $T(K) = P[C \cap K \neq \emptyset]$. Moreover, $T$ is an alternating set function that is also \textit{continuous from above} on $\mathrsfs{K}$, in the sense that $T(\bigcap_n K_n) = \lim_n T(K_n)$ for all $K_1 \supset K_2 \supset \ldots \in \mathrsfs{K}$, and every $[0,1]$-valued alternating, continuous from above set function on $\mathrsfs{K}$ is the Choquet capacity of some random closed set. 

Recall that every maxitive measure is alternating (see Proposition~\ref{prop:alt}). For a given upper-semicontinuous map $c : E \rightarrow [0, 1]$, the following construction explicitly gives a random closed set whose Choquet capacity has cardinal density $c$ \cite{Nguyen03}. Let $U$ be a uniformly distributed random variable on $[0, 1]$. Then $C = \{ x \in E : c(x) \geqslant U \}$ is a random closed set on $E$, and its Choquet capacity $T$ is maxitive and satisfies $T(K) = \bigoplus_{x\in K} c(x)$, for all $K \in \mathrsfs{K}$. 

One may observe that this random closed set is such that 
\[
C(\omega) \subset C(\omega') \mbox{ or } C(\omega') \subset C(\omega),
\]
for all $\omega, \omega' \in \mathit{\Omega}$. More generally, Miranda, Couso, and Gil \cite{Miranda04} called \textit{consonant} (of type C2) a random closed set $C$ satisfying  the above relation for all $\omega, \omega' \in \mathit{\Omega}_0$, for some event $\mathit{\Omega}_0$ of probability $1$. These authors showed that a random closed set is consonant if and only if its Choquet capacity is maxitive \cite[Corollary~5.4]{Miranda04}. 

Elements of random set theory may be found in the reference book by Matheron \cite{Matheron75}; see also the monographs by Goodman and Nguyen \cite{Goodman85} and Molchanov \cite{Molchanov05}. 
\end{example}

\begin{example}[Random sup-measures]\label{ex:random}
Let $(\mathit{\Omega}, \mathrsfs{A})$ and $(E, \mathrsfs{B})$ be measurable spaces, $P$ be a probability measure on $\mathrsfs{A}$, and $m$ be a finite $\sigma$-additive measure on $\mathrsfs{B}$. Consider a Poisson point process $(X_k, T_k)_{k \geqslant 1}$ on $\mathbb{R}_+ \times E$ with intensity $p x^{-p-1} dx \times m(dt)$, where $p > 0$. 
Then the random process defined on $\mathrsfs{B}$ by 
\[
M(B) = \bigoplus_{k \geqslant 1} X_k \cdot 1_B(T_k)
\] 
is, $\omega$ by $\omega$, a completely maxitive measure. Moreover, 
this is a \textit{$p$-Fr\'echet random sup-measure} with control measure $m$ in the sense of Stoev and Taqqu \cite[Definition~2.1]{Stoev05}, for it is a map $M : \mathit{\Omega} \times \mathrsfs{B} \rightarrow \overline{\mathbb{R}}_+$ satisfying the following axioms: 
\begin{itemize}
	\item for all $B \in \mathrsfs{B}$ the map $M(B) : \mathit{\Omega} \rightarrow \overline{\mathbb{R}}_+, \omega \mapsto M(\omega, B)$ is a random variable following a Fr\'echet distribution with shape parameter $1/p$, in such a way that, for all $x > 0$, 
	\[
	P[M(B) \leqslant x] = \exp(-m(B) x^{-p}); 
	\]
	\item for all pairwise disjoint collections $(B_j)_{j\in \mathbb{N}}$ of elements of $\mathrsfs{B}$, the random variables $M(B_j)$, $j \in \mathbb{N}$, are independent, and, almost surely, 
	\[
	M(\bigcup_{j \in \mathbb{N}} B_j) = \bigoplus_{j \in \mathbb{N}} M(B_j). 
	\]
\end{itemize}
The Poisson process $(X_k, T_k)_{k \geqslant 1}$ was introduced by de Haan \cite{deHaan84} as a tool for representing continuous-time max-stable processes. These processes play an important role in extreme value theory. 
See also Norberg \cite{Norberg86} and Resnick and Roy \cite{Resnick91} for elements on random sup-measures. 
\end{example}

\begin{example}[The home range]
Let $(X_n)_{n \geqslant 1}$ be a sequence of independent, identically distributed $\mathbb{R}^2$-valued random variables, and assume that the common distribution has compact support. We write this sequence in polar coordinates $(R_n, \mathit{\Theta}_n)_{n \geqslant 1}$. Define the map $h$ on Borel subsets $B$ of $[0, 2\pi]$ by: 
\[
h(B) = \sup \{ r \in \mathbb{R}_+ : P[R_1 > r, \mathit{\Theta}_1 \in B] > 0 \}. 
\]
Then, according to de Haan and Resnick \cite[Proposition~2.1]{deHaan94}, $h$ is a completely maxitive measure, and $h$ may be thought of as the boundary of the natural habitat of some animal, called the \textit{home range} in ecology. The sequence $(X_n)_{n \geqslant 1}$ is then seen as the successive sightings of the animal. De Haan and Resnick aimed at finding consistent estimates of the boundary $h$.  
\end{example}

The following paragraph contradicts an assertion made by van de Vel \cite[Exercise~II-3.19.1]{vanDeVel93}. 

\begin{example}[Carath\'eodory number of a convexity space]\label{ex:cexVDV}
A collection $\mathrsfs{C}$ of subsets of a set $X$ that contains $\emptyset$ and $X$ is a \textit{convexity} on $X$ if it is closed under arbitrary intersections and closed under directed unions. 
The pair $(X,\mathrsfs{C})$ is called a \textit{convexity space}, and elements of $\mathrsfs{C}$ are called \textit{convex} subsets of $X$. 
If $A \subset X$, the \textit{convex hull} $\operatorname{co}(A)$ of $A$ is the intersection of all convex subsets containing $A$. 
Advanced abstract convexity theory is developed in the monograph by van de Vel \cite{vanDeVel93}. 
The Carath\'eodory number $c(A)$ of some $A \subset X$ is the least integer $n$ such that, for each subset $B$ of $A$ and $x \in \operatorname{co}(B) \cap A$, there exists some finite subset $F$ of $B$ with cardinality $\leqslant n$ such that $x \in \operatorname{co}(F)$. 
In \cite[Exercise~II-3.19.1]{vanDeVel93}, van de Vel asserted that the map $A \mapsto c(A)$ is a maxitive (integer-valued) measure on $\mathrsfs{E}$, where $\mathrsfs{E}$ is the prepaving made up of finite unions of convex subsets of $X$. However, a simple counterexample is built as follows. Let $X$ be the three-element semilattice $\{ x_1, x_2, x_3\}$ with $x_2 = x_1 \wedge x_3$, endowed with the convexity made up of all subsets of $X$ but $\{x_1, x_3\}$. Let $A_i = \{x_i\}$ for $i = 1, 2, 3$. Then $c(A_i) = 1$ for $i = 1, 2, 3$, hence $\max_{i = 1, 2, 3} c(A_i) = 1$. However, $c(\bigcup_{i = 1, 2, 3} A_i) = c(X) = 2$, for if $B := \{x_1, x_3 \}$, one has $x_2 \in \operatorname{co}(B) \cap X = X$, while every nonempty subset $F$ of $B$ with cardinality $\leqslant 1$ is either $\{x_1\}$ or $\{x_3\}$, hence does not contain $x_2$. 
\end{example}

\begin{example}[Interpretation of maxitive measures]
Finkelstein et al.\ \cite{Finkelstein07} suggested to use maxitive measures as a model for a physicist's reasoning and beliefs about probable, possible, and impossible events. 
Kreino\-vich and Lonpr\'{e} \cite{Kreinovich05} advocated the use of maxitive measures for modelling rarity of events,  
for maxitive measures are limits of probability measures in a large deviation sense (for a justification see e.g.\ the work by O'Brien and Vervaat \cite{OBrien91}, Gerritse \cite{Gerritse96}, O'Brien \cite{OBrien96}, Akian \cite{Akian99}, Puhalskii \cite{Puhalskii94, Puhalskii01}). This interpretation is in accordance with Bouleau's criticism of extreme value theory \cite{Bouleau91}. This author noted that some events, although possible, are so rare (Bouleau gave the example of the extinction of Neanderthal Man) that they cannot be appropriately understood by classical probability theory (and in particular by extreme value theory). Since probability theory relies on the frequentist paradigm, the question of the \textit{probability} of such events would make no sense. 
For further discussion on the intuitive and the formalized distinction between \textit{probable} and \textit{possible} events, see also El Rayes and Morsi \cite[Paragraph~2]{ElRayes94} and Nguyen and Bouchon-Meunier \cite{Nguyen03}. 
\end{example}

\section{Maxitive measures as essential suprema}\label{sec:esssup}

\subsection{Introduction}\label{sec:introesssup}

In this section, we shall be interested in representing a maxitive measure $\nu$ defined on a $\sigma$-algebra $\mathrsfs{B}$ as an essential supremum with respect to some null-additive  monotone set function $\tau$, i.e.\ as  
\begin{equation}\label{eq:mu}
\nu(B) = \bigoplus_{x \in B}^{\tau} f(x), 
\end{equation}
for all $B \in \mathrsfs{B}$, as introduced in Example~\ref{ex:esssup}. 
Note that, for such a $\tau$, the set function $\delta_{\tau}$, defined by $\delta_{\tau}(B) = 1$ if $\tau(B) > 0$, $\delta_{\tau}(B) = 0$ otherwise, is a maxitive measure, and Equation~\eqref{eq:mu} is satisfied if and only if $\nu(B) = \bigoplus_{x \in B}^{\delta_{\tau}} f(x)$, for all $B \in \mathrsfs{B}$. Thus, we can restrict our attention to essential suprema with respect to maxitive measures, without loss of generality. 



\begin{definition}\label{abscont}
Let $\nu$ and $\tau$ be null-additive monotone set functions on a $\sigma$-algebra $\mathrsfs{B}$ on $E$. 
Then $\nu$ is \textit{absolutely continuous with respect to $\tau$} (or $\tau$  \textit{dominates} $\nu$), in symbols $\nu \ll \tau$, if for all $B \in \mathrsfs{B}$, $\tau(B) = 0$ implies $\nu(B) = 0$.
We shall say that $\nu$ is \textit{strongly absolutely continuous with respect to $\tau$} (or $\tau$  \textit{strongly dominates} $\nu$), in symbols $\nu \lll \tau$, if $\nu$ admits a $\mathrsfs{B}$-measurable relative density with respect to $\tau$, i.e.\ if there exists a $\mathrsfs{B}$-measurable map $f : E \rightarrow \overline{\mathbb{R}}_+$ such that Equation~\eqref{eq:mu} holds for all $B \in \mathrsfs{B}$. 
\end{definition}


Absolute continuity, although necessary in Equation~\eqref{eq:mu}, seems a priori too poor a condition for ensuring the existence of a (relative) density, i.e.\ $\nu \ll \tau$ does not imply $\nu \lll \tau$ in general. For instance, every maxitive measure $\nu$ satisfies $\nu \ll \delta_{\#}$, while $\nu$ does not necessarily have a cardinal density (see for instance Example~\ref{ex:monc} on measures of non-compactness). 
We shall understand in Section~\ref{secRN} that absolute continuity is actually a necessary and sufficient condition for the existence of a density whenever the dominating measure is \textit{$\sigma$-principal} (and the measure $\delta_{\#}$ is not $\sigma$-principal in general). 

The next proposition ensures that, under the absolute continuity condition, a relative density exists whenever a cardinal density already exists. 
Given a $\sigma$-algebra $\mathrsfs{B}$ on $E$, we say that a maxitive measure $\nu$ on $\mathrsfs{B}$ is \textit{strongly absolutely autocontinuous} (or \textit{autocontinuous} for short) if $\nu \lll \nu$. 

\begin{proposition}
Let $\nu$ be a maxitive measure on $\mathrsfs{B}$ with a $\mathrsfs{B}$-measurable cardinal density $c$. 
Then for every maxitive measure $\tau$ on $\mathrsfs{B}$, we have $\nu \ll \tau$ if and only if $\nu \lll \tau$. 
In particular, $\nu$ is autocontinuous. 
\end{proposition}

\begin{proof}
Suppose that $\nu \ll \tau$, and let us show that $\nu \lll \tau$. 
Let $B \in \mathrsfs{B}$, and let $x \in B$, $t \in \mathbb{R}_+$ such that $\tau(N) = 0$ with $N \supset B \cap \{ c > t \}$. 
If $c(x) > t$, then $x \in N$. Since $\nu \ll \tau$ and $\tau(N) = 0$, we have $\sup_{y \in N} c(y) = \nu(N) = 0$, so that $c(x) = 0$, a contradiction. Thus $c(x) \leqslant t$, and we get $\nu(B) = \bigoplus_{x\in B} c(x) \leqslant \bigoplus^\tau_{x\in B} c(x)$. 

Now we show the converse inequality. If $\nu(B)$ is infinite, this is evident. If not, let $a > \nu(B) = \bigoplus_{x\in B} c(x)$.  Then $B \cap \{ c > a \} = \emptyset$ is negligible with respect to $\tau$, hence $a \geqslant \bigoplus^\tau_{x\in B} c(x)$ by definition of essential supremum, and the result is proved. 
\end{proof}

\subsection{Existence of a relative density}

The following theorem on existence and ``uniqueness'' of relative densities is due to Barron et al.\  \cite[Theorem~3.5]{Barron00}. 
We add the following component: we define a maxitive measure $\tau$ on a $\sigma$-algebra $\mathrsfs{B}$ to be \textit{essential} if there exists a $\sigma$-finite, $\sigma$-additive measure $m$ such that $\tau(B) > 0$ if and only if $m(B) > 0$, for all $B \in \mathrsfs{B}$. 

\begin{theorem}[Barron--Cardaliaguet--Jensen]\label{radon-nikodym}
Let $\nu, \tau$ be $\sigma$-maxitive measures on $\mathrsfs{B}$. Assume that $\tau$ is essential. Then $\nu \ll \tau$ if and only if $\nu \lll \tau$. 
In this situation, the relative density of $\nu$ with respect to $\tau$ is unique $\tau$-almost everywhere. 
\end{theorem}

\begin{proof}[Sketch of the proof]
Since $\tau$ is essential we can replace, without loss of generality, $\tau$ by some $\sigma$-finite, $\sigma$-additive measure $m$ in the statement of Theorem~\ref{radon-nikodym}. 
We first assume that both $m$ and $\nu$ are finite. 
The ingenious proof given by Barron et al.\ relies on the following idea: to $\nu$ they associate the map $m_{\nu}$ defined on $\mathrsfs{B}$ by
\[
m_{\nu}(B) = \inf\left\{ \sum_{j\geqslant 1} \nu(B_j) m(B_j) : \bigcup_{j\geqslant 1} B_j = B, B_k \in \mathrsfs{B}, \forall k \geqslant 1 \right\}. 
\]
This formula is certainly inspired by the Carath\'eodory extension procedure in classical measure theory, see e.g.\ \cite[Definition~10.21]{Aliprantis06}. 
As intuition suggests, $m_{\nu}$ turns out to be a $\sigma$-additive measure, absolutely continuous with respect to $m$. Thanks to the classical Radon--Nikodym theorem there is some $\mathrsfs{B}$-measurable map $c : E \rightarrow \mathbb{R}_+$ such that 
\[
m_{\nu}(B) = \int_B c \, dm, 
\]
for all $B \in\mathrsfs{B}$, and one can prove that  
$
\nu(B) = \bigoplus_{x\in B}^{m} c(x)
$ 
for all $B \in \mathrsfs{B}$ using the following ``reconstruction'' formula for $\nu$:
\[
\nu(B) = \sup \left\{ \frac{m_{\nu}(B')}{m(B')} : B' \subset B, B'\in \mathrsfs{B}, m(B') > 0 \right\}, 
\]
for all $B\in\mathrsfs{B}$. 

Now take some (not necessarily finite) $\nu$, 
and let $\nu_1 : B \mapsto \arctan\nu(B)$. Then $\nu_1$ is a finite $\sigma$-maxitive measure, absolutely continuous with respect to $m$, hence one can write $\nu_1(B) = \bigoplus_{x\in B}^{m} c_1(x)$. 
Since $\nu_1(E) \leqslant \pi/2$, we can choose $c_1$ to be ($\mathrsfs{B}$-measurable and) such that $0 \leqslant c_1 \leqslant \pi/2$. 
It is now an easy task to show that, for all $B \in \mathrsfs{B}$, $\nu(B) = \bigoplus_{x\in B}^{m} c(x)$, where $c(x) = \tan(c_1(x))$. 

The case where $m$ is $\sigma$-finite is easily deduced. 
\end{proof}

\begin{corollary}\label{coro:esscard}
Let $\nu$ be an essential $\sigma$-maxitive measure on $\mathrsfs{B}$. Then $\nu$ is autocontinuous. 
Moreover, 
if the empty set is the only $\nu$-negligible subset, 
then $\nu$ has a cardinal density. 
\end{corollary}

Barron et al.'s theorem is interesting 
because of its proof, which points out a correspondence between $\sigma$-maxitive and $\sigma$-additive measures. However, a part of the mystery persists, for it relies on the classical Radon--Nikodym theorem: the construction of the density remains hidden. 

Note that Acerbi, Buttazzo, and Prinari \cite[Theorem~3.2]{Acerbi02} used Theorem~\ref{radon-nikodym} for resolving some non-linear minimization problems. They  considered a $\sigma$-finite, $\sigma$-additive measure $m$ on $(E, \mathrsfs{B})$, and derived sufficient conditions for a functional $F : L^{\infty}(m; E, \mathbb{R}^n) \times \mathrsfs{B} \rightarrow \mathbb{R} \cup \{ \pm\infty \}$ to be of the form  
\[
F(u, B) = \bigoplus_{x \in B}^{m} f(x, u(x)), 
\]
for some 
measurable map $f : E \times \mathbb{R}^n \rightarrow \mathbb{R} \cup \{\infty\}$ such that $f(x,\cdot)$ is lower-semicontinuous on $\mathbb{R}^n$, $m$-almost everywhere. This study was carried on by Cardaliaguet and Prinari \cite{Cardaliaguet05}, with the search for representations of the form 
\[
F(u, B) = \bigoplus_{x \in B}^{m} f(x, u(x), Du(x)), 
\]
where $u$ runs over the set of Lipschitz-continuous maps on $E$. 

Theorem~\ref{radon-nikodym} was rediscovered by Drewnowski \cite[Theorem~1]{Drewnowski09}, with a notably different proof. He applied this result to the representation of K\"{o}the function $M$-spaces as $L^{\infty}$-spaces. 
Actually, we shall see in Section~\ref{secRN} that Theorem~\ref{radon-nikodym} is a direct consequence of a more general result, proved years earlier by Sugeno and Murofushi \cite{Sugeno87}, which expresses it as a Radon--Nikodym like theorem with respect to the Shilkret integral (see Theorem~\ref{sugeno-murofushi}). 


\subsection{Maxitive measures of bounded variation}\label{par:essential}

Considering Theorem~\ref{radon-nikodym}, a natural interest is to derive sufficient conditions for a maxitive measure to be essential. A null-additive set function on $\mathrsfs{B}$ satisfies the \textit{countable chain condition} (or is \textit{CCC}) if each family of non-negligible pairwise disjoint elements of $\mathrsfs{B}$ is countable. 
(A CCC set function is sometimes called \textit{$\sigma$-decomposable}, but this terminology should be avoided, because of possible confusion with the notion of decomposability used e.g.\ by Weber \cite{Weber84}.) 
It is not difficult to show that every essential maxitive measure is CCC\@. The converse statement was the object of Mesiar's hypothesis, proposed in \cite{Mesiar97}. 
Murofushi \cite{Murofushi02} showed that this hypothesis as such is wrong, by providing a counterexample; see also Poncet \cite{Poncet07}. 
We now give the following sufficient condition for a maxitive measure to be essential. A null-additive set function $\tau$ on $\mathrsfs{B}$ is \textit{of bounded variation} if $| \tau | := \sup_{\pi} \sum_{B \in \pi} \tau(B) < \infty$, where the supremum is taken over the set of finite $\mathrsfs{B}$-partitions $\pi$ of $E$. 



\begin{proposition}\label{prop:essential} 
Every $\sigma$-maxitive measure of bounded variation on $\mathrsfs{B}$ is finite and essential. 
\end{proposition}

\begin{proof}
Let $\nu$ be a $\sigma$-maxitive measure of bounded variation on $\mathrsfs{B}$ and $m$ be the map defined on $\mathrsfs{B}$ by
\[
m(B) = \sup_{\pi} \sum_{B' \in \pi} \nu(B\cap B'),
\]
where the supremum is taken over the set of finite $\mathrsfs{B}$-partitions $\pi$ of $E$. Then $m$, called the \textit{disjoint variation} of $\nu$, is the least $\sigma$-additive measure greater than $\nu$ (see e.g.\ Pap \cite[Theorem~3.2]{Pap95}). Since $\nu$ is of bounded variation, $m$ is finite, so $\nu$ is finite. Moreover, $\nu(B) > 0$ if and only if $m(B) > 0$, so $\nu$ is essential. 
\end{proof}

\section{The idempotent integral}\label{secshi}

\subsection{Introduction}

Until today, the Lebesgue integral has given rise to many extensions. The first of them dates back to Vitali \cite{Vitali25b}, who proposed to replace $\sigma$-additive measures by some more general set functions (see the historical note by Marinacci \cite{Marinacci97}). 
In \cite{Choquet54} Choquet built on the same idea to create the tool now called the Choquet integral; it was revived by Schmeidler \cite{Schmeidler86, Schmeidler89}; its theoretical properties were developed e.g.\ by Greco \cite{Greco82}, Groes et al.\ \cite{Groes98}, K\"{o}nig \cite{Koenig03}; it has found numerous applications, as in statistics and data mining (see Murofushi and Sugeno \cite{Murofushi93b}, Grabisch \cite{Grabisch03}, Wang, Leung, and Klir \cite{Wang05}, Fallah Tehrani et al.\ \cite{FallahTehrani12}), game theory and mathematical economics (see Gilboa and Schmeidler \cite{Gilboa94}, Heilpern \cite{Heilpern02}), decision theory (see Chateauneuf \cite{Chateauneuf94}, Grabisch \cite{Grabisch95, Grabisch97}, Grabisch and Roubens \cite{Grabisch00}, Grabisch and Labreuche \cite{Grabisch10}, Mayag, Grabisch, and Labreuche \cite{Mayag11}), insurance and finance (see Chateauneuf, Kast, and Lapied \cite{Chateauneuf96}, Castagnoli, Maccheroni, and Marinacci \cite{Castagnoli04}). 

After Choquet, many authors have examined the properties of integrals where the operations $(+, \times)$ used for both the Lebesgue and the Choquet integrals are swapped for some more general pair $(\dot{+}, \dot{\times})$ of associative binary relations on $\mathbb{R}_+$ or $\overline{\mathbb{R}}_+$. 
In the case where $(\dot{+}, \dot{\times})$ is the pair $(\max, \times)$ (resp.\ $(\max, \min)$), one gets the \textit{Shilkret integral} (resp.\ \textit{Sugeno integral} or \textit{fuzzy integral}) discovered by Shilkret \cite{Shilkret71} (resp.\ by Sugeno \cite{Sugeno74}). 
For general $(\dot{+}, \dot{\times})$ various generalizations of the Lebesgue, Choquet, Shilkret, and Sugeno integrals have been introduced, including the \textit{Weber integral} \cite{Weber84, Weber86}, the \textit{pseudo-additive integral} \cite{Sugeno87}, the \textit{fuzzy t-conorm integral} \cite{Murofushi91b}, the \textit{pan integral} \cite{Yang85a}; 
see also Wang and Klir \cite{Wang92, Wang09}, Pap \cite{Pap95, Pap02b}. 
For a further generalization of all these integrals, see Sander and Siedekum \cite{Sander05c}. 

Beyond the replacement of arithmetical operations, another direction of generalization is to integrate $L$-valued functions (giving rise to $L$-valued integrals) rather than real-valued functions, where $L$ has an appropriate semiring or semimodule structure. In this process, measures can either remain real-valued if $L$ is a (semi)module (as in the Bochner integral which is a well-known extension of the Lebesgue integral, where $L$ is a Banach space), or can also be $L$-valued if $L$ is a semiring. 
Maslov \cite{Maslov87} developed an integration theory for measures with values in an ordered semiring. Other authors considered the case where $L$ is a complete lattice, see e.g.\ Greco \cite{Greco87}, Liu and Zhang \cite{Liu94}, de Cooman, Zhang, and Kerre \cite{deCooman01}, Kramosil \cite{Kramosil05a}. 
In the line of Maslov, Akian \cite{Akian99} focused on defining an integral for dioid-valued functions, and showed how crucial the assumption of \textit{continuity} of the underlying partially ordered set can be (see the monograph by Gierz et al.\ \cite{Gierz03} for background on continuous lattices and domain theory; see also \cite{Poncet12b}). 
Jonasson \cite{Jonasson98} had a similar approach, but managed to mix the powerful tool of continuous poset theory with a general ordered-semiring structure for $L$. See also Heckmann and Huth \cite{Heckmann98} for the role of continuous posets in integration theory. For extensions of the Riemann integral driven by the idea of approximation and still using arguments from continuous poset theory, see Edalat \cite{Edalat95b}, Howroyd \cite{Howroyd00}, Lawson and Lu \cite{Lawson03}, and references therein. 

A review of integration theory in mathematics should include a number of prolific developments (e.g.\ the Birkhoff integral, the Pettis integral, the stochastic It\^{o} integral, or the axiomatic approach of \textit{universal integrals} proposed by Klement, Mesiar, and Pap \cite{Klement10}, to cite only a few among many others). Needless to say this is far beyond the scope of this work; the reader may refer to the book \cite{Pap02} for a broad overview of measure and integration theory. 
In this paper, we shall limit our attention to the case where $\dot{+}$ is the maximum operation $\max = \oplus$ and $\dot{\times}$ is a pseudo-multiplication (i.e.\ a binary relation $\odot$ satisfying the properties given in Paragraph~\ref{par:odot}). This section is devoted to the construction of the related integral, that we call the \textit{idempotent $\odot$-integral}.  



\subsection{Pseudo-multiplications and their properties}\label{par:odot}

In the remaining part of this paper, we consider a binary relation $\odot$ defined on $\overline{\mathbb{R}}_+ \times \overline{\mathbb{R}}_+$ with the following properties: 
\begin{itemize}
	\item associativity; 
	\item continuity on $(0, \infty) \times [0, \infty]$; 
	\item continuity of the map $s \mapsto s \odot t$ on $(0, \infty]$, for all $t$; 
	\item monotonicity in both components; 
	\item existence of a left identity element $1_{\odot}$, i.e.\ $1_{\odot} \odot t =  t$ for all $t$; 
	\item absence of zero divisors, i.e.\ $s \odot t = 0 \Rightarrow 0 \in  \{s, t\}$, for all $s, t$; 
	\item $0$ is an annihilator, i.e.\ $0 \odot t = t \odot 0 = 0$, for all $t$. 
\end{itemize}
We call such a $\odot$ a \textit{pseudo-multiplication}. 
Pseudo-multiplications and more generally pseudo-arithmetic operations have been studied e.g.\ by Benvenuti and Mesiar \cite{Benvenuti04}. 
Note that the axioms above are stronger than in \cite{Sugeno87}, where associativity was not assumed. For more on pseudo-multiplications see also \cite{Poncet13e}. 

We consider the map $O : \overline{\mathbb{R}}_+ \rightarrow \overline{\mathbb{R}}_+$ defined by $O(t) = \inf_{s > 0} s \odot t$. 
An element $t$ of $\overline{\mathbb{R}}_+$ is \textit{$\odot$-finite} if $O(t) = 0$ (and $t$ is \textit{$\odot$-infinite} otherwise). 
We conventionally write $t \ll_{\odot} \infty$ for a $\odot$-finite element $t$. 
If $O(1_{\odot}) = 0$, we say that the pseudo-multiplication $\odot$ is \textit{non-degenerate}. This amounts to say that the set of $\odot$-finite elements differs from $\{0\}$. 



\subsection{Definition and elementary properties}\label{intshilkret}

Throughout this section, $\mathrsfs{B}$ is a $\sigma$-algebra on $E$. 
A map $f : E \rightarrow \overline{\mathbb{R}}_+$ is \textit{$\mathrsfs{B}$-measurable} 
if $\{ f > t \} \in \mathrsfs{B}$, for all $t \in \mathbb{R}_+$. 

\begin{definition}
Let $\nu$ be a maxitive measure on $\mathrsfs{B}$, and let $f : E \rightarrow \overline{\mathbb{R}}_+$ be a $\mathrsfs{B}$-measurable map. The \textit{idempotent $\odot$-integral}  
of $f$ with respect to $\nu$ is defined by 
\begin{equation}\label{defint}
\nu(f) = \dint{E} f \odot d\nu = \bigoplus_{t \in \mathbb{R}_+} t \odot \nu(f > t). 
\end{equation}
\end{definition}

The occurrence of $\infty$ in the notation $\ddint \,$ is \textit{not} an integration bound, see \cite[Theorem~I-5.7]{Poncet11} for a justification.

Generalizing Gerritse's result \cite[Proposition~3]{Gerritse96}, the following identity holds:
\[
\dint{E} f \odot d\nu = \bigoplus_{B \in \mathrsfs{B}} \left(f^{\wedge}(B) \odot \nu(B)\right),
\]
where $f^{\wedge}(A)$ stands for $\inf_{x\in A} f(x)$. Also, notice that the supremum in Equation~\eqref{defint} may be reduced to a countable supremum, for 
\begin{align*}
\dint{E} f \odot d\nu &= \bigoplus_{t \in \mathbb{R}_+} t \odot \nu\left(\bigcup_{r\in \mathbb{Q}_+, r \geqslant t}\{ f > r\}\right) 
= \bigoplus_{t \in \mathbb{R}_+} t \odot \bigoplus_{t\in \mathbb{Q}_+, r \geqslant t}\nu( f > r ) \\
&= \bigoplus_{r \in \mathbb{Q}_+} \bigoplus_{t\in \mathbb{R}_+, t \leqslant r} t \odot \nu( f > r) 
= \bigoplus_{r \in \mathbb{Q}_+} r \odot \nu( f > r), 
\end{align*}
so that Equation~\eqref{defint} is now given in a countable form. 

\begin{proposition}
Let $\nu$ be a $\sigma$-maxitive measure on $\mathrsfs{B}$. 
Then, for all $\mathrsfs{B}$-measurable maps $f : E \rightarrow \overline{\mathbb{R}}_+$, and all $r \in \mathbb{R}_+$, $B \in \mathrsfs{B}$, the following properties hold: 
\begin{itemize}
	\item $\nu(1_B) = \nu(B)$, 
	\item homogeneity: $\nu(r \odot f) = r \odot \nu(f)$, 
	\item $\sigma$-maxitivity: $\nu(\bigoplus_n f_n) = \bigoplus_n \nu(f_n)$, for every sequence of $\mathrsfs{B}$-measurable maps $f_n : E \rightarrow \overline{\mathbb{R}}_+$, 
	\item $B \mapsto \dint{B} f \odot d\nu$ is a $\sigma$-maxitive measure on $\mathrsfs{B}$.  
\end{itemize}
\end{proposition}

\begin{proof}
See Sugeno and Murofushi \cite[Proposition~6.1]{Sugeno87}. 
\end{proof}

In the case where $\odot$ is the usual multiplication $\times$, Cattaneo proved a converse statement in the sense that, given a maxitive measure $\nu$ on $\mathrsfs{B}$, the Shilkret integral $f \mapsto \dint{E} f \cdot d\nu$ is the unique scale invariant, maxitive extension of $\nu$ to the set of $\mathrsfs{B}$-measurable maps $f : E \rightarrow \mathbb{R}_+$, see \cite[Theorem~4]{Cattaneo14}, see also \cite{Cattaneo13}. 

In the case where $\odot$ is the infimum $\wedge$, it can be shown that the Sugeno integral of $f$ coincides with the distance $d_{\nu}(f, 0)$ between $f$ and $0$ with respect to the Ky Fan metric \cite{Fan44}, defined as 
\[
d_{\nu}(f, g) = \inf \{ t > 0 : \nu( | f - g | > t) \leqslant t \}.
\] 


In order to study the idempotent $\odot$-integral more deeply, 
it would be natural to fix a measurable space $(E,\mathrsfs{B})$ endowed with a $\sigma$-maxitive measure $\nu$, and, by analogy with the additive case, to look at the spaces $L^p(\nu)$, $p>0$. These spaces are defined as equivalent classes (with respect to $\nu$-almost everywhere equality) of $\mathrsfs{B}$-measurable maps $f : E \rightarrow \mathbb{R}$ such that $\| f \|_p := (\ddint |f|^p \odot d\nu)^{1/p} < \infty$; see e.g.\ Rudin \cite[Chapter~3]{Rudin87} for more background on $L^p$ spaces in the classical context of $\sigma$-additive measures. 
These are Banach spaces, as noticed by Shilkret \cite{Shilkret71} in the case where $\odot$ is the usual multiplication, and it is easily seen that the monotone and dominated convergence theorems, the Chebyshev and H\"{o}lder inequalities, etc.\ are satisfied (see \cite[Lemmata~1.4.5 and 1.4.7]{Puhalskii01} and \cite[Theorem~1.4.19]{Puhalskii01}). However, these spaces are less interesting to study than their classical counterpart, since $L^p(\nu) = L^1(\nu^{1/p})$, so that all of them can be viewed as $L^1$ spaces. In particular, $L^2(\nu)$ is not a Hilbert space. Nonetheless, these spaces can be considered as  generalizations of the spaces $L^\infty(m)$ (with $m$ a $\sigma$-additive measure), since $L^\infty(m) = L^1(\delta_m)$. 

Further properties of the Shilkret integral with respect to an optimal measure (see Definition~\ref{mesopt}) were studied by Agbeko \cite{Agbeko00} and applied to characterizations of boundedness and uniform boundedness of measurable functions. 
We also refer the reader to Puhalskii \cite{Puhalskii01} and to de Cooman \cite{deCooman97}, who both gave a pretty exhaustive treatment of the Shilkret integral. We note however that their approach is essentially limited to completely maxitive measures defined on \textit{$\tau$-algebras} (also called \textit{ample fields}, i.e.\ $\sigma$-algebras closed under arbitrary intersections, see Janssen, de Cooman, and Kerre \cite{Janssen01}); this framework has the disadvantage of breaking the parallel with classical measure theory. 
We shall come back to this debate in Section~\ref{sec:poss}.

\subsection{Examples}

We pursue the study of two examples introduced above, namely the essential supremum and the Fr\'echet random sup-measures. We also generalize the latter with the concept of regularly-varying random sup-measure. 

\begin{example}[Example~\ref{ex:esssup} continued]\label{ex:esssup2}
Let $\tau$ be a null-additive monotone set function and let $f: E \rightarrow \overline{\mathbb{R}}_+$ be some $\mathrsfs{B}$-measurable map. Then the \textit{$\tau$-essential supremum} of $f$ is the maxitive measure $\tau_f : B \mapsto \bigoplus_{x \in B}^{\tau} f(x)$; it  
can be seen as an idempotent $\odot$-integral, i.e.\ 
\[
\bigoplus_{x \in B}^{\tau} f(x) = \dint{B} f \odot d\delta_{\tau}, 
\]
where $\delta_{\tau}$ is the maxitive measure defined by $\delta_{\tau}(B) = 1$ if $\tau(B) > 0$, $\delta_{\tau}(B) = 0$ otherwise. 
Moreover, integration with respect to the $\tau$-essential supremum $\tau_f$ gives 
\[
\dint{E} g \odot d\tau_f = \bigoplus_{x \in E}^{\tau} g(x) \odot f(x) = \dint{E} g \odot f \odot d\delta_{\tau}.
\]
\end{example}


\begin{example}[Example~\ref{ex:random} continued]\label{ex:random2}
Let $(\mathit{\Omega}, \mathrsfs{A})$ and $(E, \mathrsfs{B})$ be measurable spaces, $P$ be a probability measure on $\mathrsfs{A}$, and $m$ be a finite $\sigma$-additive measure on $\mathrsfs{B}$. Let $M$ be a $p$-Fr\'echet random sup-measure with control measure $m$. 
For all $\mathrsfs{B}$-measurable maps $f : E \rightarrow \overline{\mathbb{R}}_+$, we can consider the Shilkret integral $M(f)$ defined as usual by 
\[
\dint{E} f \cdot d M = \bigoplus_{t \in \mathbb{R}_+} t \cdot M(f > t). 
\] 
This coincides with the \textit{extremal integral} of Stoev and Taqqu \cite{Stoev05} (note that these authors did not seem to know about Shilkret's or Maslov's works). It can be seen as a kind of stochastic integral with a deterministic integrand, very similar to the well-known $\alpha$-stable (or sum-stable) integral (see Samorodnitsky and Taqqu \cite{Samorodnitsky94}). 
Note that $M(f)$ is indeed a random variable, for the supremum over $\mathbb{R}_+$ can be replaced by a countable supremum (see Paragraph~\ref{intshilkret}). 
Moreover, if $f \in L^p_{+}(m)$, 
then $M(f)$ follows a Fr\'echet distribution with
\[
P[M(f) \leqslant x] = \exp(-\left\| f \right\|_{p}^{p} x^{-p}). 
\]
Here $L^p_{+}(m)$ denotes the space of equivalent classes (with respect to $m$-almost everywhere equality) of $\mathrsfs{B}$-measurable maps $f : E \rightarrow \mathbb{R}_+$ such that $\| f \|_p := (\int f^p \, dm)^{1/p} < \infty$; see Rudin \cite[Chapter~3]{Rudin87} for more background on $L^p$ spaces. 
This implies that, for every $f \in L^p_{+}(m)$, $B \mapsto \dint{B} f \cdot d M$ is itself a $p$-Fr\'echet random sup-measure with control measure $B \mapsto \int_B f^{p} \, dm$. See \cite{Stoev05} for additional properties. 
In the particular case where  
\[
M(B) = \bigoplus_{k \geqslant 1} X_k \cdot 1_B(T_k), 
\] 
for some Poisson point process $(X_k, T_k)_{k \geqslant 1}$ on $\mathbb{R}_+ \times E$ with intensity measure $p x^{-p-1} dx \times m(dt)$, we have 
\[
\dint{E} f \cdot d M = \bigoplus_{k \geqslant 1} X_k \cdot f(T_k). 
\] 
De Haan \cite{deHaan84} introduced this latter integral process and showed that, if $(X_t)_{t \in \mathbb{R}}$ is a continuous-time simple max-stable process, then there exists a Poisson process with the above properties, and a collection $(f_t)_{t \in \mathbb{R}}$ of nonnegative $L^1$ maps such that 
\[
(X_t)_{t \in \mathbb{R}} \stackrel{d}{=} (\dint{E} f_t \cdot d M), 
\]
where $\stackrel{d}{=}$ means equality in finite-dimensional distributions \cite[Theorem~3]{deHaan84}. 
\end{example}

\begin{example}[Regularly-varying sup-measures]\label{ex:random3}
A variant on the previous example can be done as follows. Let $(\mathit{\Omega}, \mathrsfs{A}, P)$ be a probability space, $(E, \mathrsfs{B})$ be a measurable space, and $m$ be a finite $\sigma$-additive measure on $\mathrsfs{B}$. We define a \textit{$p$-regularly-varying random sup-measure} with control measure $m$ to be a map $M : \mathit{\Omega} \times \mathrsfs{B} \rightarrow \overline{\mathbb{R}}_+$ satisfying the following conditions: 
\begin{itemize}
	\item for all $B \in \mathrsfs{B}$, $M(B)$ is a regularly-varying random variable of index $p$; more precisely there exists a function $L$, slowly-varying at $\infty$, such that, for all $B \in \mathrsfs{B}$, 
	\[
	P[M(B) > x] \sim m(B) x^{-p} L(x),  	
	\]
	 when $x \rightarrow \infty$; 
	\item for all pairwise disjoint collections $(B_j)_{j\in \mathbb{N}}$ of elements of $\mathrsfs{B}$, the random variables $M(B_j)$, $j \in \mathbb{N}$, are independent, and, almost surely, 
	\[
	M(\bigcup_{j \in \mathbb{N}} B_j) = \bigoplus_{j \in \mathbb{N}} M(B_j). 
	\]
\end{itemize}
Recall that $L : \mathbb{R}_+\setminus\{0\} \rightarrow \mathbb{R}_+\setminus\{0\}$ is \textit{slowly-varying at $\infty$} if, for all $a > 0$, $\lim_{x \rightarrow \infty} L(a x) / L(x) = 1$. 
See e.g.\ Resnick \cite{Resnick87} for more on regularly- and slowly-varying functions. 
For all $f \in L^p_{+}(m)$, the random variable $M(f)$ defined as the Shilkret integral of $f$ with respect to $M$ satisfies 
\[
	P[M(f) > x] \sim \left\| f \right\|_{p}^{p} x^{-p} L(x), 
\]
when $x \rightarrow \infty$. 
Let us prove this assertion. First, consider the case where $f$ is a nonnegative (measurable) simple map, i.e.\ a map of the form $f = \sum_{j=1}^k t_{j} 1_{B_j}$, where $B_1, \ldots, B_k \in \mathcal{B}$ are pairwise disjoint and $t_j > 0$ for $j = 1, \ldots, k$. One can write $f = \bigoplus_{j=1}^k t_{j} 1_{B_j}$. Thus, 
$M(f) = \bigoplus_{j=1}^k t_{j} M(B_j)$, almost surely, so that
\[
P[M(f) > x] \sim -\log P[M(f) \leqslant x] = \sum_{j=1}^k -\log P[M(B_j) \leqslant x/t_j], 
\]
since the random variables $M(B_1), \ldots, M(B_k)$ are independent. We get 
\begin{align*}
P[M(f) > x] &\sim \sum_{j=1}^k P[M(B_j) > x/t_j] (1+o(1)) \\
&= \sum_{j=1}^k m(B_j) t_j^p x^{-p} L(x/t_j) (1+o(1)) \\
&= \sum_{j=1}^k m(B_j) t_j^p x^{-p} L(x) (1+o(1)), 
\end{align*}
since $L$ is slowly-varying. This shows that $P[M(f) > x] \sim \left\|f\right\|_p^p x^{-p} L(x)$. 
In the general case where $f$ is in $L^p_{+}(m)$, let $(\varphi_n)$ be a nondecreasing sequence of nonnegative simple maps that converges pointwise to $f$. Then $\left\|\varphi_n\right\|_p \rightarrow \left\|f\right\|_p$ when $n \rightarrow \infty$. As a consequence, 
\[
P[M(\varphi_n) > x] \sim_{x \rightarrow \infty} \left\|\varphi_n\right\|_p^p x^{-p} L(x) \rightarrow_{n} \left\|f\right\|_p^p x^{-p} L(x). 
\] 
But we also have $P[M(\varphi_n) > x] \rightarrow_{n} P[M(f) > x]$, and the result follows.
\end{example}

\section{The Radon--Nikodym theorem}\label{secRN} 

\subsection{Introduction}

A widespread proof of the Radon--Nikodym theorem for $\sigma$-additive measures, due to von Neumann, uses the representation of bounded linear forms on a Hilbert space (see e.g.\ Rudin \cite[Chapter~6]{Rudin87}). But for $\sigma$-maxitive measures the space $L^2$, as already noticed, actually reduces to an $L^1$ space, for $L^2(\nu) = L^1(\nu^{1/2})$ for every $\sigma$-maxitive measure $\nu$. That is why such an approach is not possible\footnote{Actually, the really significant point in usual $L^2$ spaces is the ability to \textit{project}. Projections may still be available in ordered algebraic structures, 
see e.g.\ Cohen, Gaubert, and Quadrat \cite{Cohen04}. }, and we have to find another way for proving a Radon--Nikodym theorem for $\sigma$-maxitive measures. 
Sugeno, in relation to the Sugeno integral, was confronted with the same problem in his thesis, and gave sufficient conditions for the existence of a Radon--Nikodym derivative \cite{Sugeno74} at the cost of a topological structure on $E$. This first result was refined by Candeloro and Pucci \cite[Theorem~3.7]{Candeloro87} and Sugeno and Murofushi \cite[Corollary~8.3]{Sugeno87}. 

In this section, we give a general definition of the density of a maxitive measure with respect to the Shilkret integral. Then we recall the main theorem stating the existence of such a density \cite[Corollary~8.4]{Sugeno87}. 
Here, $\mathrsfs{B}$ still denotes a $\sigma$-algebra. 

The literature is not unanimous in the meaning of the term ``density'' applied to maxitive measures. For Akian \cite{Akian99}, a density is any map $c$ such that $\nu(\cdot) = \bigoplus_{x\in \cdot} c(x)$, i.e.\ what we called cardinal density. For Barron et al.\ \cite{Barron00} and Drewnowski \cite{Drewnowski09}, a density corresponds to our concept of relative density (see Section~\ref{sec:esssup}). 
The following definition encompasses both points of view.  
Let $\nu$ and $\tau$ be maxitive measures on $\mathrsfs{B}$. Then $\nu$ \textit{has a density with respect to} $\tau$ if there exists some $\mathrsfs{B}$-measurable map (called \textit{density}) $c : E \rightarrow \overline{\mathbb{R}}_+$ such that 
\begin{equation}\label{eq:dens}
\nu(B) = \dint{B} c \odot d\tau, 
\end{equation}
for all $B \in \mathrsfs{B}$. 

\begin{definition}\label{def:abscont}
Let $\nu$, $\tau$ be monotone set functions on $\mathrsfs{B}$. Then $\nu$ is \textit{$\odot$-absolutely continuous with respect to $\tau$} (or $\tau$ \textit{$\odot$-dominates} $\nu$), in symbols $\nu \ll_{\odot} \tau$, if for all $B \in \mathrsfs{B}$, $\nu(B) \leqslant \infty \odot \tau(B)$. 
\end{definition}

\begin{remark}
In \cite{Poncet13e}, I have given a slightly different definition of $\odot$-absolute continuity, which was that $\nu$ is \textit{$\odot$-absolutely continuous with respect to $\tau$} if for all $B \in \mathrsfs{B}$ such that $\tau(B)$ be $\odot$-finite, $\nu(B) \leqslant \infty \odot \tau(B)$. 
It is easily seen that the two definitions coincide when either $\nu$ is semi-$\odot$-finite, or $\tau$ is $\sigma$-$\odot$-finite and $\nu$ is $\sigma$-maxitive (see the definitions of semi-$\odot$-finiteness and $\sigma$-$\odot$-finiteness below). For that reason, all the results of \cite{Poncet13e} that involve the latter definition of $\odot$-absolute continuity are still valid with the former definition. 
\end{remark}

In the case where $\odot$ is the usual multiplication $\times$ (resp.\ the infimum $\wedge$), then $\ll_{\odot}$ coincides with the usual relation $\ll$ (resp.\ with $\leqslant$). 
If $\nu$ has a density with respect to $\tau$, then $\nu$ is $\odot$-absolutely continuous with respect to $\tau$, according to Definition~\ref{abscont}. 
Taking $\tau = \delta_{\#}$ in Equation~\eqref{eq:dens}, one gets $\nu(B) = \bigoplus_{x\in B} c(x)$, i.e.\ one recovers the notion of cardinal density introduced in Example~\ref{ex:denscard}. If $\mu$ is a null-additive monotone set function, then Equation~\eqref{eq:dens} with $\tau = \delta_{\mu}$ rewrites as $\nu(B) = \bigoplus_{x\in B}^{\tau} c(x)$, which fits with the case of essential suprema and relative densities introduced in Example~\ref{ex:esssup}.

\subsection{Uniqueness and finiteness of the density} 

Let $(E, \mathrsfs{B})$ be a measurable space. 
A set function $\nu : \mathrsfs{B} \rightarrow \overline{\mathbb{R}}_+$ is \textit{semi-$\odot$-finite} if, for all $B \in \mathrsfs{B}$, $\nu(B) = \bigoplus_{A \subset B} \nu(A)$, where the supremum is taken over $\{ A \in \mathrsfs{B} : A \subset B, \nu(A) \ll_{\odot} \infty \}$. 


\begin{proposition}\label{prop:semifinite}
Let $\nu$, $\tau$ be $\sigma$-maxitive measures on $\mathrsfs{B}$. 
Assume that $\nu$ is semi-$\odot$-finite and admits a $\mathrsfs{B}$-measurable density $c$ with respect to $\tau$. Then $\nu$ admits a $\odot$-finite-valued $\mathrsfs{B}$-measurable density with respect to $\tau$. 
\end{proposition}


\begin{proof}
See \cite[Proposition~3.2]{Poncet13e}. 
\end{proof}

Paralleling the classical case, we have the following result on ``uniqueness'' of the density. 

\begin{proposition}
Let $\nu$, $\tau$ be $\sigma$-maxitive measures on $\mathrsfs{B}$. If $\nu$ admits a $\mathrsfs{B}$-measurable density with respect to $\tau$, then this density is unique, $\tau$-almost everywhere. 
\end{proposition}

\begin{proof}
The assertion can be proved along the same lines as the case of the Lebesgue integral, see e.g.\ Rudin \cite[Theorem~1.39(b)]{Rudin87}. 
\end{proof}

\subsection{Principality and existence of a density}

Let $(E,\mathrsfs{B})$ be a measurable space. 
Sugeno and Murofushi \cite[Corollary~8.4]{Sugeno87} proved a Radon--Nikodym theorem for the Shilkret integral when the dominating measure is $\sigma$-$\odot$-finite and $\sigma$-principal. 

A null-additive monotone set function $\tau$ on $\mathrsfs{B}$ is \textit{$\odot$-finite} if $\tau(E) \ll_{\odot} \infty$, and \textit{$\sigma$-$\odot$-finite} if there exists some countable family $\{B_n\}_{n \in \mathbb{N}}$ of elements of $\mathrsfs{B}$ covering $E$ such that $\tau(B_n) \ll_{\odot} \infty$ for all $n$. It is \textit{$\sigma$-principal} if, for every $\sigma$-ideal $\mathrsfs{I}$ of $\mathrsfs{B}$, there exists some $L \in \mathrsfs{I}$ such that $S \setminus L$ is $\tau$-negligible, for all $S \in \mathrsfs{I}$. See \cite[Proposition~4.1]{Poncet13e} for a  justification of this terminology. 

\begin{theorem}[Sugeno--Murofushi]\label{sugeno-murofushi}
Let $\nu$, $\tau$ be $\sigma$-maxitive measures on $\mathrsfs{B}$. 
Assume that $\tau$ is $\sigma$-$\odot$-finite and $\sigma$-principal. 
Then $\nu \ll_{\odot} \tau$ if and only if there exists some $\mathrsfs{B}$-measurable map $c : E \rightarrow \overline{\mathbb{R}}_+$ such that 
\[
\nu(B) = \dint{B} c\odot d\tau, 
\]
for all $B\in \mathrsfs{B}$. 
If these conditions are satisfied, then $c$ is unique $\tau$-almost everywhere. 
Moreover, if $\nu$ is semi-$\odot$-finite, one can choose a map $c$ taking only $\odot$-finite values. 
\end{theorem}

\begin{proof}
See \cite[Theorem~8.2]{Sugeno87} 
for the original proof. 
See also \cite[Chapter~III]{Poncet11} for another proof of this theorem that makes use of order-theoretical arguments, in the case where $\odot$ is the usual multiplication. 
\end{proof}

If $\odot$ is the usual multiplication, the hypothesis of $\sigma$-$\odot$-finiteness of $\tau$ cannot be removed: consider for instance a finite set $E$, and let $\nu = \delta_{\#}$ and $\tau = \infty \cdot \delta_{\#}$ be $\sigma$-maxitive measures defined on the power set of $E$. Then $\tau$ is $\sigma$-principal and $\nu$ is absolutely continuous with respect to $\tau$, but $\nu$ never has a density with respect to $\tau$. 

Theorem~\ref{sugeno-murofushi} encompasses Theorem~\ref{radon-nikodym}, for if $\tau$ is an essential $\sigma$-max\-itive measure, then $\delta_{\tau}$ is ($\sigma$-finite and) $\sigma$-principal (use Theorem~\ref{implications}). We can thus state the following corollary. 

\begin{corollary}[Generalization of Barron--Cardaliaguet--Jensen]\label{coro:sm0}
Let $\nu, \tau$ be $\sigma$-maxitive measures on $\mathrsfs{B}$. Assume that $\tau$ is $\sigma$-principal. Then $\nu \ll \tau$ if and only if $\nu \lll \tau$. 
In this situation, the relative density of $\nu$ with respect to $\tau$ is unique $\tau$-almost everywhere. 
\end{corollary}

We have another simple consequence, which generalizes Corollary~\ref{coro:esscard}. 

\begin{corollary}\label{coro:sm}
Let $\nu$ be a $\sigma$-principal $\sigma$-maxitive measure on $\mathrsfs{B}$. Then $\nu$ is autocontinuous. 
Moreover, if the empty set is the only $\nu$-negligible subset, 
then $\nu$ is completely maxitive (and has a cardinal density). 
\end{corollary}

\begin{proof}
Simply take $\tau = \delta_{\nu}$ in the previous theorem.
\end{proof}



At this stage we think it useful to recall the characterization of those $\sigma$-maxitive measures $\tau$ with the \textit{Radon--Nikodym property}, i.e.\ such that all $\sigma$-maxitive measures $\odot$-dominated by $\tau$ have a measurable density with respect to $\tau$. 


\begin{theorem}\label{thm:rncarac}
Given a non-degenerate pseudo-multiplication $\odot$, a $\sigma$-max\-itive measure $\tau$ on $\mathrsfs{B}$ satisfies the Radon--Nikodym property with respect to the idempotent $\odot$-integral if and only if $\tau$ is $\sigma$-$\odot$-finite and $\sigma$-principal. 
\end{theorem}

\begin{proof}
See \cite{Poncet13e}. 
\end{proof}

\begin{corollary}
Let $\tau$ be a $\sigma$-maxitive measure on $\mathrsfs{B}$. Then $\tau$ satisfies the Radon--Nikodym property with respect to the Shilkret integral if and only if $\tau$ is $\sigma$-finite and $\sigma$-principal. 
\end{corollary}

\begin{corollary}
Let $\tau$ be a $\sigma$-maxitive measure on $\mathrsfs{B}$. Then $\tau$ satisfies the Radon--Nikodym property with respect to the Sugeno integral if and only if $\tau$ is $\sigma$-principal. 
\end{corollary}

Two $\sigma$-maxitive measures $\nu$ and $\tau$ on $\mathrsfs{B}$ are \textit{associated} if there exists a third $\sigma$-maxitive measure $\mu$ on $\mathrsfs{B}$ such that $\nu \lll \mu$ and $\tau \lll \mu$. 
A reformulation of Corollary~\ref{coro:sm0} is that, if $\tau$ is $\sigma$-principal and $\nu \ll \tau$, then $\nu$ and $\tau$ are associated. 
With this notion of associated maxitive measures we can give a variant of the Radon--Nikodym type theorem, which is a generalization of Puhalskii \cite[Theorem~1.6.34]{Puhalskii01} and de Cooman \cite[Theorem~7.2]{deCooman97}. 

\begin{theorem}[Idempotent Radon--Nikodym theorem, variant]
Let $\odot$ be a pseudo-multiplication that makes $\overline{\mathbb{R}}_+$ into an exact residual semigroup (see Section~\ref{sec:residual} in the appendix). 
Let $\nu$, $\tau$ be $\sigma$-maxitive measures on $\mathrsfs{B}$, and assume that $\nu$ and $\tau$ are associated. 
Then $\nu \ll_{\odot} \tau$ if and only if there exists some $\mathrsfs{B}$-measurable map $c : E \rightarrow \overline{\mathbb{R}}_+$ such that 
\[
\nu(B) = \dint{B} c\odot d\tau, 
\]
for all $B\in \mathrsfs{B}$. 
If these conditions are satisfied, then $c$ is unique $\tau$-almost everywhere. 
Moreover, if $\nu$ is semi-$\odot$-finite, one can choose a map $c$ taking only $\odot$-finite values. 
\end{theorem}

\begin{proof}
We assume that $\nu$ and $\tau$ are associated and such that $\nu \ll_{\odot} \tau$. 
By definition, there is a $\sigma$-maxitive measure $\mu$ on $\mathrsfs{B}$ such that $\nu \lll \mu$ and $\tau \lll \mu$. 
So there are $\mathrsfs{B}$-measurable maps $c_1, c_2 : E \rightarrow \overline{\mathbb{R}}_+$ such that $\nu(B) = \bigoplus^{\mu}_{x\in B} c_1(x)$ and  $\tau(B) = \bigoplus^{\mu}_{x\in B} c_2(x)$, for all $B \in \mathrsfs{B}$. 

We use the notations of Section~\ref{sec:residual} in the appendix. 
Let $A$ be the subset 
\[
A = \{ x \in E : c_1(x) \not\ll_{\odot} c_2(x) \}. 
\]
We show that $A$ is $\mu$-negligible. 
We have 
\begin{align*}
A &= \{ x \in E : c_1(x) > \infty \odot c_2(x) \} \\
&= \bigcup_{q \in \mathbb{Q}_+} \{ x \in E : c_1(x) > q \mbox{ and } q \geqslant \infty \odot c_2(x) \} \\
&= \bigcup_{q \in \mathbb{Q}_+} B_q \cap \{ c_1 > q \}, 
\end{align*}
where $B_q$ is the subset $\{ x \in E : \infty \odot c_2(x) \leqslant q \}$. 
Notice that $B_q$ is $\mathrsfs{B}$-measurable since 
\[
B_q = \bigcap_{r \in \mathbb{Q}_+} \{ x \in E : r \odot c_2(x) \leqslant q \}, 
\]
and hence $A$ is $\mathrsfs{B}$-measurable too. 
To prove that $A$ is $\mu$-negligible first note that 
\[
\infty \odot \tau(B_q) = \bigoplus^{\mu}_{x\in B_q} \infty \odot c_2(x) \leqslant q, 
\]
for all $q \in \mathbb{Q}_+$. 
Since $\nu \ll_{\odot} \tau$ this implies $\nu(B_q) \leqslant q$ for all $q \in \mathbb{Q}_+$. 
Since $\nu(B_q)$ is the $\mu$-essential supremum of $c_1$ on $B_q$, i.e.\ 
\[
\nu(B_q) = \inf \{ t > 0 : \mu(B_q \cap \{ c_1 > t \}) = 0 \}, 
\]
this shows that $\mu(B_q \cap \{ c_1 > q \}) = 0$. 
Consequently, 
\[
\mu(A) = \bigoplus_{q \in \mathbb{Q}_+} \mu(B_q \cap \{ c_1 > t \}) = 0. 
\]
By definition of $A$, we have $c_1(x) \ll_{\odot} c_2(x)$ for all $x \in E \setminus A$, so we can define the map $c : E \rightarrow \overline{\mathbb{R}}_+$ by $c(x) = 0$ if $x \in A$ and $c(x) = (c_1(x) / c_2(x))_{\odot}$ if $x \in E \setminus A$ (see again Section~\ref{sec:residual} for the notations). 
The map $c$ is $\mathrsfs{B}$-measurable because  
\begin{align*}
\{ x \in E : c(x) \leqslant t \} &= A \cup \{ x \in E \setminus A : (c_1(x) / c_2(x))_{\odot} \leqslant t \} \\
&= A \cup \{ x \in E \setminus A : c_1(x) \leqslant t \odot c_2(x) \}, 
\end{align*}
for all $t \in \mathbb{R}_+$. 
By assumption $(\overline{\mathbb{R}}_+, \odot)$ is exact, so $c_1(x) = c(x) \odot c_2(x)$ for all $x \in E \setminus A$. 
As a consequence, 
\begin{align*}
\nu(B) &= \dint{B} c_1(x) \odot d\delta_{\mu} \\
&= \dint{B \cap (E \setminus A)} c_1(x) \odot d\delta_{\mu} \\
&= \dint{B \cap (E \setminus A)} c(x) \odot c_2(x) \odot d\delta_{\mu} \\
&= \dint{B} c(x) \odot c_2(x) \odot d\delta_{\mu} \\
&= \dint{B} c(x) \odot d\tau, 
\end{align*}
for all $B \in \mathrsfs{B}$, and the result is proved. 
\end{proof}

\section{Optimality of maxitive measures}\label{sec:optim}

\subsection{Definition of optimal measures}

In this section we focus on the special case of \textit{optimal measures}. 
Let $(E, \mathrsfs{B})$ be a measurable space. 
A set function $\nu$ on $\mathrsfs{B}$ is \textit{continuous from above} 
if $\nu(B) = \lim_n \nu(B_n)$, for all $B_1 \supset B_2 \supset \ldots \in \mathrsfs{B}$ such that $B = \bigcap_n B_n$ 
(we do not impose the condition $\nu(B_{n_0}) < \infty$ for some $n_0$). 
A monotone null-additive set function that is both continuous from above and from below is a \textit{fuzzy measure}. 
Continuity from above is automatically satisfied for finite $\sigma$-additive measures, but this is untrue for (finite) $\sigma$-maxitive measures (see Puri and Ralescu \cite{Puri82} for a counterexample, see also Wang and Klir \cite[Example~3.13]{Wang92}), so special care is needed. 
The following definition is given by Agbeko \cite{Agbeko95}. 

\begin{definition}\label{mesopt}
An \textit{optimal measure} is a maxitive fuzzy measure. 
\end{definition}

Surprisingly, it suffices for a maxitive measure to be continuous from above in order to satisfy continuity from below:

\begin{proposition}[Murofushi--Sugeno--Agbeko]
A set function $\nu$ on $\mathrsfs{B}$ is an optimal measure if and only if it is a continuous from above maxitive  measure. In this case, for all sequences $(B_n)$ of elements of $\mathrsfs{B}$,
\[
\nu(\bigcup_{n\in \mathbb{N}} B_n) = \max_{n\in\mathbb{N}} \nu(B_n),
\]
where the $\max$ operator signifies that the supremum is reached. 
\end{proposition}

\begin{proof}
Murofushi and Sugeno \cite{Murofushi93} and after them Agbeko \cite[Lemma~1.4]{Agbeko95} and Kramosil \cite{Kramosil06a} showed that every continuous from above maxitive measure $\nu$ satisfies the identity of the proposition; the first part of the proposition is then an easy consequence.
\end{proof}


The property of continuity from above in Definition~\ref{mesopt} is thus a strong condition. It becomes even more obvious with the following result. It was proved by Agbeko \cite[Theorem~1.2]{Agbeko95} using Zorn's lemma, and Fazekas \cite[Theorem~9]{Fazekas97} supplied an elementary proof. To formulate it, recall first that a \textit{$\nu$-atom} (called \textit{indecomposable $\nu$-atom} by Agbeko) is an element $H$ of $\mathrsfs{B}$ such that $\nu(H) > 0$, and for each $B\in\mathrsfs{B}$ either $\nu(H \setminus B) = 0$, or $\nu(H \cap B) = 0$.

\begin{theorem}[Agbeko--Fazekas]\label{thmagbeko}
Let $\nu$ be an optimal measure on $\mathrsfs{B}$. Then there exists an at most countable collection $(H_n)_{n \in N}$ of pairwise disjoint $\nu$-atoms $H_n \in \mathrsfs{B}$ such that
\begin{equation}\label{agbekofazekas}
\nu(B) = \max_{n\in N} \nu(B \cap H_n), 
\end{equation}
for all $B \in \mathrsfs{B}$, where the $\max$ operator signifies that the supremum is reached. In particular, $\nu$ takes an at most countable number of values. 
\end{theorem}

A consequence of this theorem is that every optimal measure takes an at most countable number of values. 

An optimal measure $\nu$ satisfies the \textit{exhaustivity} property, according to the terminology used by Pap \cite{Pap95}, 
 i.e.\ $\nu(B_n) \rightarrow 0$ when $n \rightarrow \infty$ for all pairwise disjoint $B_1, B_2, \ldots \in \mathrsfs{B}$. In fact, exhaustivity is exactly what a $\sigma$-maxitive measure needs to be optimal: 

\begin{proposition}\label{prop:exhaus}
A $\sigma$-maxitive measure is optimal if and only if it is exhaustive. 
\end{proposition}

\begin{proof}
The easy proof is left to the reader.
\end{proof}

Optimal measures were also studied (under various names) by Rie\v{c}a\-nov\'{a} \cite{Riecanova84}, Murofushi and Sugeno \cite{Murofushi93}, Arslanov and Ismail \cite{Arslanov04}. In particular, the last-mentioned authors proved that the cardinality of some nonempty set $E$ is non-measurable\footnote{A cardinal $|E|$ is \textit{measurable} if there exists a two-valued probability measure on $2^E$ making all singletons negligible. The existence of measurable cardinals remains an open question. } 
if and only if all optimal measures on $2^E$ have a cardinal density \cite[Theorem~19]{Arslanov04}. 
In \cite{Poncet12b} we studied $L$-valued optimal measures defined on the Borel algebra of a topological space, where $L$ is a partially ordered set. 

In Section~\ref{secRN} we introduced \textit{semi-$\odot$-finiteness} for maxitive measures. For optimal measures, this merely reduces to $\odot$-finiteness. 

\begin{proposition}
An optimal measure is semi-$\odot$-finite if and only if it is $\odot$-finite. 
\end{proposition}

\begin{proof}
Let $\nu$ be an optimal measure on $\mathrsfs{B}$. 
If $\nu$ is $\odot$-finite, it is clearly semi-$\odot$-finite. 
Conversely, suppose that $\nu$ is semi-$\odot$-finite. 
If $\nu(E) = 0$, then $\nu$ is $\odot$-finite. 
Otherwise, let $0 < s < \nu(E)$. 
By semi-$\odot$-finiteness, $\nu(E)$ is the supremum of $\{ \nu(B) : B \in \mathrsfs{B}, \nu(B) \ll_{\odot} \infty \}$, so $\nu(E)$ is also the supremum of $\{ \nu(B) : B \in \mathrsfs{B}, s < \nu(B) \ll_{\odot} \infty \}$. 
In view of Fazekas \cite[Remark~5]{Fazekas97}, the latter subset is finite, so its supremum is a maximum. 
This shows in particular that $\nu(E) \ll_{\odot} \infty$, i.e.\ that $\nu$ is $\odot$-finite. 
\end{proof}









\subsection{Densities of optimal measures}

In this paragraph, we use previous results on the existence of densities for $\sigma$-maxitive measures, and apply them to optimal measures. 



Agbeko proved Theorem~\ref{sugeno-murofushi} independently of Sugeno and Murofushi \cite{Sugeno87} in the particular case where $\tau$ is a normed optimal measure and $\nu$ is a finite optimal measure on $\mathrsfs{B}$ \cite[Theorem~2.4]{Agbeko95}. This is indeed a particular case thanks to  \cite[Lemma~2.1]{Murofushi93}, which states that every optimal measure is CCC, hence $\sigma$-principal under Zorn's lemma. Below we show without Zorn's lemma that every optimal measure is $\sigma$-principal (hence CCC by \cite[Proposition~4.1]{Poncet13e}). We actually show the stronger result that every optimal measure is essential. 

\begin{proposition}\label{prop:optsig}
Every optimal measure is essential (hence $\sigma$-principal, hence CCC and autocontinuous). 
\end{proposition}

\begin{proof}
Let $\nu$ be an optimal measure on a $\sigma$-algebra $\mathrsfs{B}$, and let $(H_n)_{n \in N}$ be a collection satisfying the conditions of the Agbeko-Fazekas Theorem (Theorem~\ref{thmagbeko}).  
We can suppose, without loss of generality, that $\nu$ is finite. 
We define $m$ on $\mathrsfs{B}$ by 
\[
m(B) = \sum_n \nu(B \cap H_n).  
\]
Then one can show that $m$ is a $\sigma$-finite, $\sigma$-additive measure on $\mathrsfs{B}$ such that $m(B) > 0$ if and only if $\nu(B) > 0$. 
What makes $m$ additive is that $\nu((B \cup B') \cap H_n) = \nu(B \cap H_n) + \nu(B' \cap H_n)$ whenever $B \cap B' = \emptyset$. This is because, if $B \cap B' = \emptyset$, then $\nu(B \cap H_n) > 0$ implies $\nu(B' \cap H_n) = 0$, since $\nu(H_n) = \nu(H_n \setminus (B \cap B')) = \nu(H_n \setminus B) \oplus \nu(H_n \setminus B') = \nu(H_n \setminus B') > 0$. 
\end{proof}

However, an optimal measure is not of bounded variation in general, as the next proposition shows. 
Recall that $| \nu |$ denotes the supremum of $\{\sum_{B \in \pi} \nu(B) : \pi \mbox{ is a finite $\mathrsfs{B}$-partition of $E$ } \}$. 

\begin{proposition}
For every optimal measure $\nu$ we have $| \nu | = \sum_n \nu(H_n)$, where $(H_n)_{n \in N}$ is a collection satisfying the conditions of Theorem~\ref{thmagbeko}. 
In particular, $\nu$ is of bounded variation if and only if $\sum_n \nu(H_n) < \infty$. 
\end{proposition}

\begin{proof}
Let $\nu$ be an optimal measure on a $\sigma$-algebra $\mathrsfs{B}$, and let $(H_n)_{n \in N}$ be a collection satisfying the conditions of the Agbeko-Fazekas Theorem (Theorem~\ref{thmagbeko}).  

Recall that $| \nu |$ is defined as 
$
| \nu | = \sup_{\pi} \sum_{B \in \pi} \nu(B)
$, 
where the supremum is taken over the set of finite $\mathrsfs{B}$-partitions $\pi$ of $E$. 
Let $\pi_n$ denote the finite $\mathrsfs{B}$-partition $\{ H_1, \ldots, H_n, E \setminus \bigcup_{k=1}^n H_k \}$. 
Then $\sum_{k=1}^n \nu(H_k) \leqslant \sum_{k=1}^n \nu(H_k) + \nu(\bigcap_{k=1}^n E \setminus H_k) \leqslant | \nu |$, so that $\sum_{k=1}^{\infty} \nu(H_k) \leqslant | \nu |$. 

Conversely, let $\{ B_1, \ldots, B_n \}$ be a finite $\mathrsfs{B}$-partition of $E$. We can suppose without loss of generality that $\nu(B_k) > 0$ for all $1 \leqslant k \leqslant n$. By the Agbeko--Fazekas Theorem, for every $k = 1, \ldots, n$ there exists some $n_k$ such that $0 < \nu(B_k) = \nu(B_k \cap H_{n_k}) \leqslant \nu(H_{n_k})$. Moreover, $k \neq k'$ implies $n_k \neq n_{k'}$, because if $H := H_{n_k} = H_{n_{k'}}$ and $k \neq k'$, then $B_k \cap B_{k'} = \emptyset$, so $\nu(H) = \nu(H \setminus (B_k \cap B_{k'})) = \nu(H \setminus B_k) \oplus \nu(H \setminus B_{k'}) = 0$, a contradiction. 
Consequently, $\sum_{k=1}^n \nu(B_k) \leqslant \sum_{k=1}^n \nu(H_{n_k}) \leqslant \sum_{k=1}^{\infty} \nu(H_k)$, so that $| \nu | \leqslant \sum_{k=1}^{\infty} \nu(H_k)$. 
\end{proof}


As a consequence of Proposition~\ref{prop:optsig}, we derive the Radon--Nikodym like theorem for optimal measures due to Agbeko. 

\begin{corollary}[Agbeko]
Let $\nu$, $\tau$ be $\sigma$-maxitive measures on $\mathrsfs{B}$. Assume that $\tau$ is $\odot$-finite and optimal. Then $\nu \ll_{\odot} \tau$ if and only if there exists some $\mathrsfs{B}$-measurable map $c : E \rightarrow \overline{\mathbb{R}}_+$ such that  
\[
\nu(B) = \dint{B} c \odot d\tau, 
\]
for all $B\in \mathrsfs{B}$. If these conditions are satisfied, then $c$ is unique $\tau$-almost everywhere. 
\end{corollary}

\begin{proof}
Combine Theorem~\ref{sugeno-murofushi} and Proposition~\ref{prop:optsig}, or use Agbeko \cite[Theorem~2.4]{Agbeko95} for the original statement. 
\end{proof}

\begin{problem}
Characterize those $\sigma$-maxitive measures $\tau$ that satisfy the \textit{optimal} Radon--Nikodym property, i.e.\ such that all optimal measures that are $\odot$-absolutely continuous with respect to $\tau$, have a measurable density with respect to $\tau$. 
\end{problem}





\section{A novel definition for possibility measures}\label{sec:poss}


\subsection{Towards an appropriate definition of possibility measures}

 

Possibility theory can be treated as an analogue of probability theory, where probability measures are replaced by their maxitive counterpart. 
This point of view has been developed over the last few years by several authors including Bellalouna \cite{Bellalouna92}, Akian, Quadrat, and Viot \cite{Akian94b, Akian98}, Akian \cite{Akian95}, Del Moral and Doisy \cite{delMoral98}, de Cooman \cite{deCooman95b, deCooman97, deCooman97b, deCooman97c}, Puhalskii \cite{Puhalskii01}, Barron, Cardaliaguet, and Jensen \cite{Barron03}, Fleming \cite{Fleming04} among others. See also Baccelli et al.\ \cite{Baccelli92}. 
Analogies with probability theory, especially stressed by de Cooman \cite{deCooman95b} and Akian et al.\ \cite{Akian98}, arise in the definitional aspects (such as the notion of independent events, or the concept of \textit{maxingale} which replaces that of martingale \cite{Puhalskii01, Barron03}) as well as in important results  such as the law of large numbers or the central limit theorem. Nonetheless, possibility theory has its own specificities, for instance the surprising fact that convergence in ``possibility'' implies almost sure convergence\footnote{Recall that probabilists are familiar with the converse implication. } (see \cite[Proposition~28]{Akian95} and \cite[Theorem~1.3.5]{Puhalskii01}). 

In a stochastic context, the Radon--Nikodym property is highly desirable if one wants to dispose of conditional laws. 
In the $\sigma$-additive case this property is achieved by the classical Radon--Nikodym theorem\footnote{Notice that every probability measure is  $\sigma$-principal, see Theorem~\ref{implications} in the Appendix.}, but in the $\sigma$-maxitive case this property may fail in absence of the $\sigma$-principality condition. 
To overcome this drawback, most of the publications require the possibility measure under study $\mathit{\Pi}$ to be completely maxitive, i.e.\ to have a cardinal density, thus to be of the form  
\begin{equation}\label{eq:possdens}
\mathit{\Pi}[A] = \bigoplus_{\omega \in A} c(\omega). 
\end{equation}
This condition was imposed by Akian et al.\ \cite{Akian94b, Akian98}, Akian \cite{Akian95}, Del Moral and Doisy \cite{delMoral98}, de Cooman \cite{deCooman95b, deCooman97, deCooman97b, deCooman97c}, Puhalskii \cite{Puhalskii01}, Fleming \cite{Fleming04}. 
Hypothesis~\eqref{eq:possdens} then facilitates the definition of conditioning, for $\mathit{\Pi}[X|Y]$ can be defined by the data of its cardinal density $c_{X|Y}$ given by: 
\[
c_{X|Y}(x|y) = \frac{c_{(X, Y)}(x, y)}{c_Y(y)}, 
\]
if $c_Y(y) > 0$, and $c_{X|Y}(x|y) = 0$ otherwise, where $c_X$ and $c_Y$ are the respective (maximal) cardinal densities of $\mathit{\Pi}_X := \mathit{\Pi} \circ X^{-1}$ and $\mathit{\Pi}_Y$, and $c_{(X,Y)}$ that of the random variable $(X,Y) : \mathit{\Omega} \times \mathit{\Omega} \rightarrow \mathbb{R}_+$. 
In \cite{deCooman01} and \cite{Puhalskii01}, another restrictive hypothesis was adopted, for their authors only considered completely maxitive measures defined on $\tau$-algebras. A $\tau$-algebra $\mathrsfs{A}$ on $\mathit{\Omega}$ being atomic, every $\omega \in \mathit{\Omega}$ is contained in a smallest event, denoted by $[\omega]_{\mathrsfs{A}}$. This particularity enables one to give an explicit formula of conditional laws, $\omega$ by $\omega$. 

The assumption of complete maxitivity and the use of $\tau$-algebras instead of $\sigma$-algebras, if easier to handle, are not satisfactory in the situation where one wants to parallel probability theory. A different framework is possible, and we suggest to adopt the following definition of a possibility measure. 



\begin{definition}\label{def:poss}
Let $(\mathit{\Omega}, \mathrsfs{A})$ be a measurable space. 
A \textit{possibility measure} (or a \textit{possibility} for short) on $(\mathit{\Omega}, \mathrsfs{A})$ is a $\sigma$-principal $\sigma$-maxitive measure $\mathit{\Pi}$ on $\mathrsfs{A}$ such that $\mathit{\Pi}[\mathit{\Omega}]= 1$. Then $(\mathit{\Omega}, \mathrsfs{A}, \mathit{\Pi})$ is called a \textit{possibility space}.
\end{definition}

\subsection{Conditional law with respect to a possibility measure}

A conjunction of factors tends to confirm that this is the right definition. 
Firstly, properties of $\mathit{\Pi}$ are transferred to the ``laws'' of random variables. 
If $(E, \mathrsfs{B})$ is a measurable space and $X : \mathit{\Omega} \rightarrow E$ is a random variable, its (possibility) law $\mathit{\Pi}_X$ on $\mathrsfs{B}$ is the possibility measure defined by $\mathit{\Pi}_X(B) = \mathit{\Pi}[X \in B] := \mathit{\Pi}[X^{-1}(B)]$. Moreover, if $\mathit{\Pi}$ is optimal (resp.\ completely maxitive), then $\mathit{\Pi}_X$ is optimal (resp.\ completely maxitive).

Secondly, the $\sigma$-principality property ensures that the Radon--Nikodym property is satisfied for the idempotent $\odot$-integral  $\mathit{\Sigma}[X] := \ddint X \odot d\mathit{\Pi}$ 
of some random variable $X : \mathit{\Omega} \rightarrow \mathbb{R}_+$. Thus, following the classical approach of Halmos and Savage \cite{Halmos49}, conditioning can be defined as follows. 
Let $X : \mathit{\Omega} \rightarrow \mathbb{R}_+$ be a random variable and $\mathrsfs{F}$ be a sub-$\sigma$-algebra of $\mathrsfs{A}$. 
The  $\sigma$-maxitive measure defined on $\mathrsfs{F}$ by $A \mapsto \mathit{\Sigma}[X \odot 1_A] = \dint{A} X \odot d\mathit{\Pi}$ is absolutely continuous with respect to the possibility $\mathit{\Pi} |_{\mathrsfs{F}}$. Thus, there exists some $\mathrsfs{F}$-measurable random variable from $\mathit{\Omega}$ into $\mathbb{R}_+$, written $\mathit{\Sigma}[X | \mathrsfs{F}]$, 
such that $\mathit{\Sigma}[X \odot 1_A] = \mathit{\Sigma}\left[ \mathit{\Sigma}[X|\mathrsfs{F}] \odot 1_A \right]$ for all $A \in\mathrsfs{F}$.


Barron et al.\ \cite{Barron03} considered the special case $\mathit{\Pi} := \delta_P$, where $P$ is a probability measure. Then $\mathit{\Pi}$ is essential, hence $\sigma$-principal, so it is a possibility measure in the sense of Definition~\ref{def:poss}, and the integral $\mathit{\Sigma}[X]$ of a nonnegative random variable $X$ coincides with the $P$-essential supremum of $X$, i.e.\ $\mathit{\Sigma}[X] = \bigoplus_{\omega \in \mathit{\Omega}}^{P} X(\omega)$. Also, whenever $\mathit{\Sigma}[X] < \infty$, one has $\mathit{\Sigma}[X|\mathrsfs{F}] = \lim_{p \rightarrow \infty} E[X^p | \mathrsfs{F}]^{1/p}$, $P$-almost surely (where $E[X]$ denotes the usual expected value of $X$ with respect to the probability measure $P$), see \cite[Proposition~2.12]{Barron03}. Barron et al.\ derived a number of properties  that still work in our more general context, as asserted by the next result. 





\begin{proposition}
Let $\mathrsfs{F}$ be a sub-$\sigma$-algebra of $\mathrsfs{A}$, and let $X, X', Y : \mathit{\Omega} \rightarrow \mathbb{R}_+$ be nonnegative random variables with $Y$ $\mathrsfs{F}$-measurable. Then the following assertions hold: 
\begin{enumerate}
	\item\label{bcj1} $Y = \mathit{\Sigma}[X | \mathrsfs{F}]$ a.e.\ if and only if
	$
	\mathit{\Sigma}[X \odot Z] = \mathit{\Sigma}[Y \odot Z]
	$ 
	for all nonnegative $\mathrsfs{F}$-measurable random variables $Z$; 
	\item\label{bcj3} if $X \leqslant Y$ a.e., then $\mathit{\Sigma}[X | \mathrsfs{F}] \leqslant Y$ a.e.; 
	\item\label{bcj4} if $\lambda \geqslant 0$, then $\mathit{\Sigma}[X \oplus (\lambda \odot X') |\mathrsfs{F}] = \mathit{\Sigma}[X |\mathrsfs{F}] \oplus (\lambda \odot \mathit{\Sigma}[X' |\mathrsfs{F}])$ a.e.; 
	\item\label{bcj5} $\mathit{\Sigma}\left[ \mathit{\Sigma}[X |\mathrsfs{F}] \right] = \mathit{\Sigma}[X]$; 
	\item\label{bcj6} if $X$ is $\mathrsfs{F}$-measurable then $\mathit{\Sigma}[X | \mathrsfs{F}] = X$ a.e., 
\end{enumerate}
where ``a.e.'' stands for ``$\mathit{\Pi}$-almost everywhere''. 
\end{proposition}

\begin{proof}
Note that if $X_1 = X_2$ a.e.\ and $X_2 = X_3$ a.e.\ then $X_1 = X_3$ a.e. 

\eqref{bcj1} By definition $\mathit{\Sigma}[X \odot 1_A] = \mathit{\Sigma}\left[ \mathit{\Sigma}[X|\mathrsfs{F}] \odot 1_A \right]$ for all $A \in\mathrsfs{F}$. Since $Z$ equals $\bigoplus_{q \in \mathbb{Q}_+} q \odot 1_{Z > q}$, for every nonnegative $\mathrsfs{F}$-measurable random variable $Z$, we obtain $\mathit{\Sigma}[X \odot Z] = \mathit{\Sigma}\left[ \mathit{\Sigma}[X|\mathrsfs{F}] \odot Z \right]$. So if $Y = \mathit{\Sigma}[X|\mathrsfs{F}]$ a.e., then $\mathit{\Sigma}[X \odot Z] = \mathit{\Sigma}[Y \odot Z]$. 
Conversely, suppose that $\mathit{\Sigma}[X \odot Z] = \mathit{\Sigma}[Y \odot Z]$ for every nonnegative $\mathrsfs{F}$-measurable random variable $Z$. 
Then $\mathit{\Sigma}[\mathit{\Sigma}[ X | \mathrsfs{F} ] \odot 1_A] = \mathit{\Sigma}[Y \odot 1_A]$ for all $A \in \mathrsfs{F}$. 
By Theorem~\ref{sugeno-murofushi} this implies that $\mathit{\Sigma}[ X | \mathrsfs{F} ] = Y$ a.e. 


\eqref{bcj4} Let $Y'$ be the $\mathrsfs{F}$-measurable random variable equal to $\mathit{\Sigma}[X |\mathrsfs{F}] \oplus (\lambda \odot \mathit{\Sigma}[X' |\mathrsfs{F}])$. 
If $Z$ is a nonnegative $\mathrsfs{F}$-measurable random variable then 
\begin{align*}
\mathit{\Sigma}[ Y' \odot Z ] &= \mathit{\Sigma}\left[ (\mathit{\Sigma}[X |\mathrsfs{F}] \odot Z) \oplus (\lambda \odot \mathit{\Sigma}[X' |\mathrsfs{F}] \odot Z) \right] \\
&= \mathit{\Sigma}[ \mathit{\Sigma}[X |\mathrsfs{F}] \odot Z ] \oplus (\lambda \odot \mathit{\Sigma}[ \mathit{\Sigma}[X' |\mathrsfs{F}] \odot Z ]) \\
&= \mathit{\Sigma}[ X \odot Z] \oplus (\lambda \odot \mathit{\Sigma}[ X' \odot Z]) \mbox{, by \eqref{bcj1}, }\\
&= \mathit{\Sigma}[(X \oplus (\lambda \odot X')) \odot Z]. 
\end{align*}
So by \eqref{bcj1} we obtain $Y' = \mathit{\Sigma}[X \oplus (\lambda \odot X') | \mathrsfs{F}]$ a.e. 

\eqref{bcj5} By definition $\mathit{\Sigma}[X \odot 1_A] = \mathit{\Sigma}\left[ \mathit{\Sigma}[X|\mathrsfs{F}] \odot 1_A \right]$ for all $A \in\mathrsfs{F}$. So taking in particular $A = \mathit{\Omega}$ we get $\mathit{\Sigma}[X] = \mathit{\Sigma}\left[\mathit{\Sigma}[X|\mathrsfs{F}]\right]$. 

\eqref{bcj6} This is a direct consequence of \eqref{bcj1}. 

\eqref{bcj3} If $X \leqslant Y$ a.e., then $X \oplus Y = Y$ a.e., hence $\mathit{\Sigma}[ (X \oplus Y) \odot Z] = \mathit{\Sigma}[Y \odot Z]$
 for every nonnegative $\mathrsfs{F}$-measurable random variable $Z$. So by \eqref{bcj1} this shows that $Y = \mathit{\Sigma}[ X \oplus Y | \mathrsfs{F}]$ a.e.
Consequently, 
\begin{align*}
\mathit{\Sigma}[X | \mathrsfs{F}] \oplus Y &= \mathit{\Sigma}[X | \mathrsfs{F}] \oplus \mathit{\Sigma}[Y | \mathrsfs{F}] \mbox{ a.e., by \eqref{bcj6}, } \\
&= \mathit{\Sigma}[X \oplus Y | \mathrsfs{F}] \mbox{ a.e., by \eqref{bcj4}, }\\
&= Y \mbox{ a.e., } 
\end{align*}
so $\mathit{\Sigma}[X | \mathrsfs{F}] \leqslant Y$ a.e. 
\end{proof}


From these properties, Barron et al.\ deduced an ergodic theorem for maxima and, with the concept of maxingales, developed a theory of optimal stopping in $L^\infty$. 

Our new perspective on possibility measures should encourage us to recast possibility theory. The next step would be to see whether convergence theorems given in \cite{Akian95} and \cite{Puhalskii01} remain unchanged.

\section{Conclusion and perspectives}\label{seccon}

In this paper, we have emphasized the link between essential suprema representations and Radon--Nikodym like theorems for the idempotent $\odot$-integral.  
We have shown that the Radon--Nikodym type theorem proved by Sugeno and Murofushi encompasses similar results including those of Agbeko, Barron et al., Drewnowski. 
We have proved a variant of this theorem that generalizes results due to de Cooman, Puhalskii. 
We have also recalled a converse statement to the Sugeno--Murofushi theorem, i.e.\ the characterization of those $\sigma$-maxitive measures satisfying the Radon--Nikodym property as being $\sigma$-$\odot$-finite $\sigma$-principal. 





\begin{acknowledgements}
I am grateful to Colas Bardavid who carefully read a preliminary version of the manuscript and made very accurate suggestions. 
I wish to thank Marianne Akian who made useful remarks and provided a counterexample to \cite[Exercise~II-3.19.1]{vanDeVel93} inserted as Example~\ref{ex:cexVDV}, and Jimmie D.\ Lawson for his advice and comments. 
I also thank two anonymous referees who pointed out some missing references in the original manuscript. 
\end{acknowledgements}

\bibliographystyle{plain}

\begin{thebibliography}{100}

\bibitem{Acerbi02}
Emilio Acerbi, Giuseppe Buttazzo, and Francesca Prinari.
\newblock The class of functionals which can be represented by a supremum.
\newblock {\em J. Convex Anal.}, 9(1):225--236, 2002.

\bibitem{Agbeko95}
Nutefe~Kwami Agbeko.
\newblock On the structure of optimal measures and some of its applications.
\newblock {\em Publ. Math. Debrecen}, 46(1-2):79--87, 1995.

\bibitem{Agbeko00}
Nutefe~Kwami Agbeko.
\newblock How to characterize some properties of measurable functions.
\newblock {\em Math. Notes (Miskolc)}, 1(2):87--98, 2000.

\bibitem{Akian95}
Marianne Akian.
\newblock Theory of cost measures: convergence of decision variables.
\newblock Rapport de recherche 2611, INRIA, France, 1995.

\bibitem{Akian99}
Marianne Akian.
\newblock Densities of idempotent measures and large deviations.
\newblock {\em Trans. Amer. Math. Soc.}, 351(11):4515--4543, 1999.

\bibitem{Akian94b}
Marianne Akian, Jean-Pierre Quadrat, and Michel Viot.
\newblock Bellman processes.
\newblock In {\em Proceedings of the 11th International Conference on Analysis
  and Optimization of Systems held at Sophia Antipolis, June 15--17, 1994},
  volume 199 of {\em Lecture Notes in Control and Information Sciences}, pages
  302--311, Berlin, 1994. Springer-Verlag.
\newblock Edited by Guy Cohen and Jean-Pierre Quadrat.

\bibitem{Akian98}
Marianne Akian, Jean-Pierre Quadrat, and Michel Viot.
\newblock Duality between probability and optimization.
\newblock In {\em Idempotency (Bristol, 1994)}, volume~11 of {\em Publ. Newton
  Inst.}, pages 331--353, Cambridge, 1998. Cambridge University Press.

\bibitem{Aliprantis06}
Charalambos~D. Aliprantis and Kim~C. Border.
\newblock {\em Infinite dimensional analysis}.
\newblock Springer, Berlin, third edition, 2006.
\newblock A hitchhiker's guide.

\bibitem{Appell05}
J\"urgen Appell.
\newblock Lipschitz constants and measures of noncompactness of some
  pathological maps arising in nonlinear fixed point and eigenvalue theory.
\newblock In {\em Proceedings of the Conference on Function Spaces,
  Differential Operators and Nonlinear Analysis held at Praha, May 27 - June 1,
  2004}, pages 19--27. Math. Inst. Acad. Sci. of Czech Republic, 2005.
\newblock Edited by P. Dr\'abek, J. R\'akosnik.

\bibitem{Arslanov04}
M.~Z. Arslanov and E.~E. Ismail.
\newblock On the existence of a possibility distribution function.
\newblock {\em Fuzzy Sets and Systems}, 148(2):279--290, 2004.

\bibitem{Baccelli92}
Fran\c cois~L. Baccelli, Guy Cohen, Geert~Jan Olsder, and Jean-Pierre Quadrat.
\newblock {\em Synchronization and linearity}.
\newblock Wiley Series in Probability and Mathematical Statistics. John Wiley
  \& Sons Ltd., Chichester, 1992.
\newblock An algebra for discrete event systems.

\bibitem{Barron00}
Emmanuel~N. Barron, Pierre Cardaliaguet, and Robert~R. Jensen.
\newblock Radon--{N}ikodym theorem in {$L\sp \infty$}.
\newblock {\em Appl. Math. Optim.}, 42(2):103--126, 2000.

\bibitem{Barron03}
Emmanuel~N. Barron, Pierre Cardaliaguet, and Robert~R. Jensen.
\newblock Conditional essential suprema with applications.
\newblock {\em Appl. Math. Optim.}, 48(3):229--253, 2003.

\bibitem{Bellalouna92}
Faouzi Bellalouna.
\newblock {\em Un point de vue lin\'{e}aire sur la programmation dynamique.
  D\'{e}tection de ruptures dans le cadre des probl\`{e}mes de fiabilit\'{e}.}
\newblock PhD thesis, {U}niversit\'{e} {P}aris-{IX} {D}auphine, {F}rance, 1992.

\bibitem{Benvenuti04}
Pietro Benvenuti and Radko Mesiar.
\newblock Pseudo-arithmetical operations as a basis for the general measure and
  integration theory.
\newblock {\em Inform. Sci.}, 160(1-4):1--11, 2004.

\bibitem{Bernhard00}
Pierre Bernhard.
\newblock Max-plus algebra and mathematical fear in dynamic optimization.
\newblock {\em Set-Valued Anal.}, 8(1-2):71--84, 2000.
\newblock Set-valued analysis in control theory.

\bibitem{Bouleau91}
Nicolas Bouleau.
\newblock Splendeurs et mis\`{e}res des lois de valeurs extr\^{e}mes.
\newblock {\em Risques}, 4:85--92, 1991.

\bibitem{Candeloro87}
Domenico Candeloro and Sabrina Pucci.
\newblock Radon--{N}ikodym derivatives and conditioning in fuzzy measure
  theory.
\newblock {\em Stochastica}, 11(2-3):107--120, 1987.

\bibitem{Cardaliaguet05}
Pierre Cardaliaguet and Francesca Prinari.
\newblock Supremal representation of {$L\sp \infty$} functionals.
\newblock {\em Appl. Math. Optim.}, 52(2):129--141, 2005.

\bibitem{Castagnoli04}
Erio Castagnoli, Fabio Maccheroni, and Massimo Marinacci.
\newblock {C}hoquet insurance pricing: a caveat.
\newblock {\em Math. Finance}, 14(3):481--485, 2004.

\bibitem{Cattaneo13}
Marco~E.G.V. Cattaneo.
\newblock On maxitive integration.
\newblock Technical Report 147, University of Munich, Germany, 2013.

\bibitem{Cattaneo14}
Marco~E.G.V. Cattaneo.
\newblock Maxitive integral of real-valued functions.
\newblock In {\em Proceedings of the 15th {I}nternational {C}onference ({IPMU}
  2014) held in {M}ontpellier, {F}rance, {J}uly 15--19, 2014, {P}art {I}},
  volume 442 of {\em Communications in Computer and Information Science}, pages
  226--235. Springer, 2014.
\newblock Edited by {A}nne {L}aurent, {O}livier {S}trauss, {B}ernadette
  {B}ouchon-{M}eunier, {R}onald {R}. {Y}ager.

\bibitem{Cerda07}
Joan Cerd{\`a}.
\newblock Lorentz capacity spaces.
\newblock In {\em Interpolation theory and applications}, volume 445 of {\em
  Contemp. Math.}, pages 45--59. Amer. Math. Soc., Providence, RI, 2007.

\bibitem{Chateauneuf94}
Alain Chateauneuf.
\newblock Modeling attitudes towards uncertainty and risk through the use of
  {C}hoquet integral.
\newblock {\em Ann. Oper. Res.}, 52:3--20, 1994.
\newblock Decision theory and decision systems.

\bibitem{Chateauneuf96}
Alain Chateauneuf, Robert Kast, and Andr\'e Lapied.
\newblock {C}hoquet pricing for financial markets with frictions.
\newblock {\em Math. Finance}, 6(3):323--330, 1996.

\bibitem{Choquet54}
Gustave Choquet.
\newblock Theory of capacities.
\newblock {\em Ann. Inst. Fourier, Grenoble}, 5:131--295, 1953--1954.

\bibitem{Cohen04}
Guy Cohen, St\'ephane Gaubert, and Jean-Pierre Quadrat.
\newblock Duality and separation theorems in idempotent semimodules.
\newblock {\em Linear Algebra Appl.}, 379:395--422, 2004.
\newblock Tenth Conference of the International Linear Algebra Society.

\bibitem{deCooman95b}
Gert de~Cooman.
\newblock The formal analogy between possibility and probability theory.
\newblock In {\em Foundations and applications of possibility theory,
  {P}roceedings of the {I}nternational {W}orkshop ({FAPT} '95) held in {G}hent,
  December 13--15, 1995}, volume~8 of {\em Advances in Fuzzy
  Systems---Applications and Theory}, pages 71--87, River Edge, NJ, 1995. World
  Scientific Publishing Co. Inc.
\newblock Edited by Gert de Cooman, Da Ruan and Etienne E. Kerre.

\bibitem{deCooman97}
Gert de~Cooman.
\newblock Possibility theory. {I}. {T}he measure- and integral-theoretic
  groundwork.
\newblock {\em Internat. J. Gen. Systems}, 25(4):291--323, 1997.

\bibitem{deCooman97b}
Gert de~Cooman.
\newblock Possibility theory. {II}. {C}onditional possibility.
\newblock {\em Internat. J. Gen. Systems}, 25(4):325--351, 1997.

\bibitem{deCooman97c}
Gert de~Cooman.
\newblock Possibility theory. {III}. {P}ossibilistic independence.
\newblock {\em Internat. J. Gen. Systems}, 25(4):353--371, 1997.

\bibitem{deCooman01}
Gert de~Cooman, Guangquan Zhang, and Etienne~E. Kerre.
\newblock Possibility measures and possibility integrals defined on a complete
  lattice.
\newblock {\em Fuzzy Sets and Systems}, 120(3):459--467, 2001.

\bibitem{deHaan84}
Laurens de~Haan.
\newblock A spectral representation for max-stable processes.
\newblock {\em Ann. Probab.}, 12(4):1194--1204, 1984.

\bibitem{deHaan94}
Laurens de~Haan and Sidney~I. Resnick.
\newblock Estimating the home range.
\newblock {\em J. Appl. Probab.}, 31(3):700--720, 1994.

\bibitem{delMoral98}
Pierre Del~Moral and Michel Doisy.
\newblock Maslov idempotent probability calculus. {I}.
\newblock {\em Teor. Veroyatnost. i Primenen.}, 43(4):735--751, 1998.

\bibitem{Doty05}
Dave Doty, Xiaoyang Gu, Jack~H. Lutz, Elvira Mayordomo, and Philippe Moser.
\newblock Zeta-dimension.
\newblock In {\em Proceedings of the Thirtieth International Symposium on
  Mathematical Foundations of Computer Science (Gdansk, Poland, August 29 -
  September 2, 2005)}, pages 283--294. Springer-Verlag, 2005.

\bibitem{Drewnowski09}
Lech Drewnowski.
\newblock A representation theorem for maxitive measures.
\newblock {\em Indag. Math. (N.S.)}, 20(1):43--47, 2009.

\bibitem{Dubois88}
Didier Dubois and Henri Prade.
\newblock {\em Possibility theory}.
\newblock Plenum Press, New York, {F}rench edition, 1988.
\newblock An approach to computerized processing of uncertainty, With the
  collaboration of Henri Farreny, Roger Martin-Clouaire and Claudette
  Testemale, translated by E. F. Harding, with a foreword by Lotfi A. Zadeh.

\bibitem{Dubois11}
Didier Dubois and Henri Prade.
\newblock Possibility theory and its applications: Where do we stand?
\newblock {\em Mathware and Soft Computing}, 18(1):18--31, 2011.

\bibitem{Dubois12}
Didier Dubois and Henri Prade.
\newblock Possibility theory.
\newblock In {\em Computational complexity. {V}olumes 1--6}, pages 2240--2252.
  Springer, New York, 2012.

\bibitem{Edalat95b}
Abbas Edalat.
\newblock Domain theory and integration.
\newblock {\em Theoret. Comput. Sci.}, 151(1):163--193, 1995.
\newblock Topology and completion in semantics (Chartres, 1993).

\bibitem{ElRayes94}
Ahmed~Bahaa El-Rayes and Nehad~N. Morsi.
\newblock Generalized possibility measures.
\newblock {\em Inform. Sci.}, 79(3-4):201--222, 1994.

\bibitem{Falconer90}
Kenneth Falconer.
\newblock {\em Fractal geometry}.
\newblock John Wiley \& Sons Ltd., Chichester, 1990.
\newblock Mathematical foundations and applications.

\bibitem{FallahTehrani12}
Ali Fallah~Tehrani, Weiwei Cheng, Krzysztof Dembczy{\'n}ski, and Eyke
  H{\"u}llermeier.
\newblock Learning monotone nonlinear models using the {C}hoquet integral.
\newblock {\em Mach. Learn.}, 89(1-2):183--211, 2012.

\bibitem{Fan44}
Ky~Fan.
\newblock Entfernung zweier zuf\"alligen {G}r\"ossen und die {K}onvergenz nach
  {W}ahrscheinlichkeit.
\newblock {\em Math. Z.}, 49:681--683, 1944.

\bibitem{Fazekas97}
Istv\'an Fazekas.
\newblock A note on ``optimal measures''.
\newblock {\em Publ. Math. Debrecen}, 51(3-4):273--277, 1997.

\bibitem{Finkelstein07}
Andrei~M. Finkelstein, Olga Kosheleva, Tanja Magoc, Erik Madrid, Scott~A.
  Starks, and Julio Urenda.
\newblock To properly reflect physicists reasoning about randomness, we also
  need a maxitive (possibility) measure.
\newblock {\em J. Uncertain Systems}, 1(2):84--108, 2007.

\bibitem{Fleming04}
Wendell~H. Fleming.
\newblock Max-plus stochastic processes.
\newblock {\em Appl. Math. Optim.}, 49(2):159--181, 2004.

\bibitem{Gelman97}
Boris~D. Gel{\cprime}man.
\newblock Topological properties of the set of fixed points of multivalued
  mappings.
\newblock {\em Mat. Sb.}, 188(12):33--56, 1997.

\bibitem{Gerritse96}
Bart Gerritse.
\newblock Varadhan's theorem for capacities.
\newblock {\em Comment. Math. Univ. Carolin.}, 37(4):667--690, 1996.

\bibitem{Gierz03}
Gerhard Gierz, Karl~Heinrich Hofmann, Klaus Keimel, Jimmie~D. Lawson,
  Michael~W. Mislove, and Dana~S. Scott.
\newblock {\em Continuous lattices and domains}, volume~93 of {\em Encyclopedia
  of Mathematics and its Applications}.
\newblock Cambridge University Press, Cambridge, 2003.

\bibitem{Gilboa94}
Itzhak Gilboa and David Schmeidler.
\newblock Additive representations of non-additive measures and the {C}hoquet
  integral.
\newblock {\em Ann. Oper. Res.}, 52:43--65, 1994.
\newblock Decision theory and decision systems.

\bibitem{Goodman85}
Irwin~R. Goodman and Hung~T. Nguyen.
\newblock {\em Uncertainty models for knowledge-based systems}.
\newblock North-Holland Publishing Co., Amsterdam, 1985.
\newblock A unified approach to the measurement of uncertainty, With a foreword
  by Madan M. Gupta.

\bibitem{Grabisch95}
Michel Grabisch.
\newblock Fuzzy integral in multicriteria decision making.
\newblock {\em Fuzzy Sets and Systems}, 69(3):279--298, 1995.

\bibitem{Grabisch97}
Michel Grabisch.
\newblock Fuzzy measures and integrals for decision making and pattern
  recognition.
\newblock {\em Tatra Mt. Math. Publ.}, 13:7--34, 1997.
\newblock Fuzzy structures.

\bibitem{Grabisch03}
Michel Grabisch.
\newblock Modelling data by the {C}hoquet integral.
\newblock In {\em Information fusion in data mining}, volume 123 of {\em
  Studies in Fuzziness and Soft Computing}, pages 135--148. Springer-Verlag
  Berlin Heidelberg, 2003.

\bibitem{Grabisch10}
Michel Grabisch and Christophe Labreuche.
\newblock A decade of application of the {C}hoquet and {S}ugeno integrals in
  multi-criteria decision aid.
\newblock {\em Ann. Oper. Res.}, 175:247--286, 2010.

\bibitem{Grabisch00}
Michel Grabisch and Marc Roubens.
\newblock Application of the {C}hoquet integral in multicriteria decision
  making.
\newblock In {\em Fuzzy measures and integrals}, volume~40 of {\em Stud.
  Fuzziness Soft Comput.}, pages 348--374. Physica, Heidelberg, 2000.

\bibitem{Greco82}
Gabriele~H. Greco.
\newblock On the representation of functionals by means of integrals.
\newblock {\em Rend. Sem. Mat. Univ. Padova}, 66:21--42, 1982.

\bibitem{Greco87}
Gabriele~H. Greco.
\newblock Fuzzy integrals and fuzzy measures with their values in complete
  lattices.
\newblock {\em J. Math. Anal. Appl.}, 126(2):594--603, 1987.

\bibitem{Groes98}
Ebbe Groes, J{\o}rgen Jacobsen, Birgitte Sloth, and Torben Tran{\ae}s.
\newblock Axiomatic characterizations of the {C}hoquet integral.
\newblock {\em Econom. Theory}, 12(2):441--448, 1998.

\bibitem{Halmos49}
Paul~R. Halmos and Leonard~J. Savage.
\newblock Application of the {R}adon--{N}ikodym theorem to the theory of
  sufficient statistics.
\newblock {\em Ann. Math. Statistics}, 20:225--241, 1949.

\bibitem{Harding97}
John Harding, Massimo Marinacci, Nhu~T. Nguyen, and Tonghui Wang.
\newblock Local {R}adon--{N}ikodym derivatives of set functions.
\newblock {\em Internat. J. Uncertain. Fuzziness Knowledge-Based Systems},
  5(3):379--394, 1997.

\bibitem{Heckmann98}
Reinhold Heckmann and Michael Huth.
\newblock Quantitative semantics, topology, and possibility measures.
\newblock {\em Topology Appl.}, 89(1-2):151--178, 1998.
\newblock Domain theory.

\bibitem{Heilpern02}
Stanislaw Heilpern.
\newblock Using {C}hoquet integral in economics.
\newblock {\em Statist. Papers}, 43(1):53--73, 2002.
\newblock Choquet integral and applications.

\bibitem{Howroyd00}
John~D. Howroyd.
\newblock A domain-theoretic approach to integration in {H}ausdorff spaces.
\newblock {\em LMS J. Comput. Math.}, 3:229--273 (electronic), 2000.

\bibitem{Janssen01}
Hugo~J. Janssen, Gert de~Cooman, and Etienne~E. Kerre.
\newblock Ample fields as a basis for possibilistic processes.
\newblock {\em Fuzzy Sets and Systems}, 120(3):445--458, 2001.

\bibitem{Jonasson98}
Johan Jonasson.
\newblock On positive random objects.
\newblock {\em J. Theoret. Probab.}, 11(1):81--125, 1998.

\bibitem{Klement10}
Erich~Peter Klement, Radko Mesiar, and Endre Pap.
\newblock A universal integral as common frame for {C}hoquet and {S}ugeno
  integral.
\newblock {\em IEEE Trans. Fuzzy Syst.}, 18(1):178--187, 2010.

\bibitem{Kolokoltsov89a}
Vassili~N. Kolokoltsov and Victor~P. Maslov.
\newblock Idempotent analysis as a tool of control theory and optimal
  synthesis. {I}.
\newblock {\em Funktsional. Anal. i Prilozhen.}, 23(1):1--14, 1989.

\bibitem{Kolokoltsov89b}
Vassili~N. Kolokoltsov and Victor~P. Maslov.
\newblock Idempotent analysis as a tool of control theory and optimal
  synthesis. {II}.
\newblock {\em Funktsional. Anal. i Prilozhen.}, 23(4):53--62, 1989.

\bibitem{Koenig03}
Heinz K{\"o}nig.
\newblock The (sub/super)additivity assertion of {C}hoquet.
\newblock {\em Studia Math.}, 157(2):171--197, 2003.

\bibitem{Kramosil05a}
Ivan Kramosil.
\newblock Generalizations and extensions of lattice-valued possibilistic
  measures, part {I}.
\newblock Technical Report 952, Institute of Computer Science, Academy of
  Sciences of the Czech Republic, 2005.

\bibitem{Kramosil06a}
Ivan Kramosil.
\newblock Generalizations and extensions of lattice-valued possibilistic
  measures, part {II}.
\newblock Technical Report 985, Institute of Computer Science, Academy of
  Sciences of the Czech Republic, 2006.

\bibitem{Kraetschmer03}
Volker Kr{\"a}tschmer.
\newblock When fuzzy measures are upper envelopes of probability measures.
\newblock {\em Fuzzy Sets and Systems}, 138(3):455--468, 2003.

\bibitem{Kreinovich05}
Vladik Kreinovich and Luc Longpr\'e.
\newblock Kolmogorov complexity leads to a representation theorem for
  idempotent probabilities ($\sigma$-maxitive measures).
\newblock {\em ACM SIGACT News}, 36(3):107--112, 2005.

\bibitem{Lawson03}
Jimmie~D. Lawson and Bin Lu.
\newblock Riemann and {E}dalat integration on domains.
\newblock {\em Theoret. Comput. Sci.}, 305(1-3):259--275, 2003.
\newblock Topology in computer science (Schlo{\ss} Dagstuhl, 2000).

\bibitem{Liu94}
Xue~Cheng Liu and Guangquan Zhang.
\newblock Lattice-valued fuzzy measure and lattice-valued fuzzy integral.
\newblock {\em Fuzzy Sets and Systems}, 62(3):319--332, 1994.

\bibitem{Lutz03}
Jack~H. Lutz.
\newblock The dimensions of individual strings and sequences.
\newblock {\em Inform. and Comput.}, 187(1):49--79, 2003.

\bibitem{Lutz05}
Jack~H. Lutz.
\newblock Effective fractal dimensions.
\newblock {\em MLQ Math. Log. Q.}, 51(1):62--72, 2005.

\bibitem{Mallet-Paret10a}
John Mallet-Paret and Roger~D. Nussbaum.
\newblock Inequivalent measures of noncompactness.
\newblock {\em Annali di Matematica Pura ed Applicata}, 48, 2010.

\bibitem{Mallet-Paret10b}
John Mallet-Paret and Roger~D. Nussbaum.
\newblock Inequivalent measures of noncompactness and the radius of the
  essential spectrum.
\newblock {\em Proc. Amer. Math. Soc.}, 139(3):917--930, 2011.

\bibitem{Marinacci97}
Massimo Marinacci.
\newblock Vitali's early contribution to non-additive integration.
\newblock {\em Riv. Mat. Sci. Econom. Social.}, 20(2):153--158, 1997.

\bibitem{Maslov87}
Victor~P. Maslov.
\newblock {\em M\'ethodes op\'eratorielles}.
\newblock \'Editions Mir, Moscow, 1987.
\newblock Translated from the Russian by Djilali Embarek.

\bibitem{Matheron75}
Georges Matheron.
\newblock {\em Random sets and integral geometry}.
\newblock John Wiley\thinspace \&\thinspace Sons, New York-London-Sydney, 1975.
\newblock With a foreword by Geoffrey S. Watson, Wiley Series in Probability
  and Mathematical Statistics.

\bibitem{Mayag11}
Brice Mayag, Michel Grabisch, and Christophe Labreuche.
\newblock A representation of preferences by the {C}hoquet integral with
  respect to a 2-additive capacity.
\newblock {\em Theory and Decision}, 71(3):297--324, 2011.

\bibitem{Mesiar95}
Radko Mesiar.
\newblock On the integral representation of fuzzy possibility measures.
\newblock {\em International Journal of General Systems}, 23(2):109--121, 1995.

\bibitem{Mesiar97}
Radko Mesiar.
\newblock Possibility measures, integration and fuzzy possibility measures.
\newblock {\em Fuzzy Sets and Systems}, 92(2):191--196, 1997.

\bibitem{Mesiar99}
Radko Mesiar and Endre Pap.
\newblock Idempotent integral as limit of {$g$}-integrals.
\newblock {\em Fuzzy Sets and Systems}, 102(3):385--392, 1999.
\newblock Fuzzy measures and integrals.

\bibitem{Miranda04}
Enrique Miranda, In{\'e}s Couso, and Pedro Gil.
\newblock A random set characterization of possibility measures.
\newblock {\em Inform. Sci.}, 168(1-4):51--75, 2004.

\bibitem{Molchanov05}
Ilya~S. Molchanov.
\newblock {\em Theory of random sets}.
\newblock Probability and its Applications (New York). Springer-Verlag London
  Ltd., London, 2005.

\bibitem{Murofushi02}
Toshiaki Murofushi.
\newblock Two-valued possibility measures induced by {$\sigma$}-finite
  {$\sigma$}-additive measures.
\newblock {\em Fuzzy Sets and Systems}, 126(2):265--268, 2002.

\bibitem{Murofushi91b}
Toshiaki Murofushi and Michio Sugeno.
\newblock Fuzzy {$t$}-conorm integral with respect to fuzzy measures:
  generalization of {S}ugeno integral and {C}hoquet integral.
\newblock {\em Fuzzy Sets and Systems}, 42(1):57--71, 1991.

\bibitem{Murofushi93}
Toshiaki Murofushi and Michio Sugeno.
\newblock Continuous-from-above possibility measures and {$f$}-additive fuzzy
  measures on separable metric spaces: characterization and regularity.
\newblock {\em Fuzzy Sets and Systems}, 54(3):351--354, 1993.

\bibitem{Murofushi93b}
Toshiaki Murofushi and Michio Sugeno.
\newblock Some quantities represented by the {C}hoquet integral.
\newblock {\em Fuzzy Sets and Systems}, 56(2):229--235, 1993.

\bibitem{Nagata83}
Jun-iti Nagata.
\newblock {\em Modern dimension theory}, volume~2 of {\em Sigma Series in Pure
  Mathematics}.
\newblock Heldermann Verlag, Berlin, revised edition, 1983.

\bibitem{Nguyen03}
Hung~T. Nguyen and Bernadette Bouchon-Meunier.
\newblock Random sets and large deviations principle as a foundation for
  possibility measures.
\newblock {\em Soft Comput.}, 8:61--70, 2003.

\bibitem{Norberg86}
Tommy Norberg.
\newblock Random capacities and their distributions.
\newblock {\em Probab. Theory Relat. Fields}, 73(2):281--297, 1986.

\bibitem{OBrien96}
George~L. O'Brien.
\newblock Sequences of capacities, with connections to large-deviation theory.
\newblock {\em J. Theoret. Probab.}, 9(1):19--35, 1996.

\bibitem{OBrien90}
George~L. O'Brien, Paul J. J.~F. Torfs, and Wim Vervaat.
\newblock Stationary self-similar extremal processes.
\newblock {\em Probab. Theory Related Fields}, 87(1):97--119, 1990.

\bibitem{OBrien91}
George~L. O'Brien and Wim Vervaat.
\newblock Capacities, large deviations and loglog laws.
\newblock In {\em Stable processes and related topics (Ithaca, NY, 1990)},
  volume~25 of {\em Progr. Probab.}, pages 43--83, Boston, MA, 1991.
  Birkh\"auser Boston.

\bibitem{Pap95}
Endre Pap.
\newblock {\em Null-additive set functions}, volume 337 of {\em Mathematics and
  its Applications}.
\newblock Kluwer Academic Publishers Group, Dordrecht, 1995.

\bibitem{Pap02}
Endre Pap, editor.
\newblock {\em Handbook of measure theory. {V}ol. {I}, {II}}.
\newblock North-Holland, Amsterdam, 2002.

\bibitem{Pap02b}
Endre Pap.
\newblock Pseudo-additive measures and their applications.
\newblock In {\em Handbook of measure theory, {V}ol. {II}}, pages 1403--1468.
  North-Holland, Amsterdam, 2002.

\bibitem{Poncet07}
Paul Poncet.
\newblock A note on two-valued possibility ({$\sigma$}-maxitive) measures and
  {M}esiar's hypothesis.
\newblock {\em Fuzzy Sets and Systems}, 158(16):1843--1845, 2007.

\bibitem{Poncet10}
Paul Poncet.
\newblock A decomposition theorem for maxitive measures.
\newblock {\em Linear Algebra Appl.}, 435(7):1672--1680, 2011.

\bibitem{Poncet11}
Paul Poncet.
\newblock {\em Infinite-dimensional idempotent analysis: the role of continuous
  posets}.
\newblock PhD thesis, \'Ecole Polytechnique, Palaiseau, France, 2011.

\bibitem{Poncet12b}
Paul Poncet.
\newblock How regular can maxitive measures be?
\newblock {\em Topology Appl.}, 160(4):606--619, 2013.

\bibitem{Poncet13e}
Paul Poncet.
\newblock The idempotent {R}adon--{N}ikodym theorem has a converse statement.
\newblock {\em Inform. Sci.}, 271:115--124, 2014.

\bibitem{Puhalskii94}
Anatolii~A. Puhalskii.
\newblock On the theory of large deviations.
\newblock {\em Theory Probab. Appl.}, 38(3):490--497, 1994.

\bibitem{Puhalskii01}
Anatolii~A. Puhalskii.
\newblock {\em Large deviations and idempotent probability}, volume 119 of {\em
  Chapman \& Hall/CRC Monographs and Surveys in Pure and Applied Mathematics}.
\newblock Chapman \& Hall/CRC, Boca Raton, FL, 2001.

\bibitem{Puri82}
Madan~L. Puri and Dan~A. Ralescu.
\newblock A possibility measure is not a fuzzy measure.
\newblock {\em Fuzzy Sets and Systems}, 7(3):311--313, 1982.

\bibitem{Resnick87}
Sidney~I. Resnick.
\newblock {\em Extreme values, regular variation, and point processes},
  volume~4 of {\em Applied Probability. A Series of the Applied Probability
  Trust}.
\newblock Springer-Verlag, New York, 1987.

\bibitem{Resnick91}
Sidney~I. Resnick and Rishin Roy.
\newblock Random usc functions, max-stable processes and continuous choice.
\newblock {\em Ann. Appl. Probab.}, 1(2):267--292, 1991.

\bibitem{Riecanova84}
Zdena Rie{\v{c}}anov{\'a}.
\newblock Regularity of semigroup-valued set functions.
\newblock {\em Math. Slovaca}, 34(2):165--170, 1984.

\bibitem{Rudin87}
Walter Rudin.
\newblock {\em Real and complex analysis}.
\newblock McGraw-Hill Book Co., New York, 3 edition, 1987.

\bibitem{Samorodnitsky94}
Gennady Samorodnitsky and Murad~S. Taqqu.
\newblock {\em Stable non-{G}aussian random processes}.
\newblock Stochastic Modeling. Chapman \& Hall, New York, 1994.
\newblock Stochastic models with infinite variance.

\bibitem{Sander05c}
Wolfgang Sander and Jens Siedekum.
\newblock Multiplication, distributivity and fuzzy-integral. {III}.
\newblock {\em Kybernetika (Prague)}, 41(4):497--518, 2005.

\bibitem{Schmeidler86}
David Schmeidler.
\newblock Integral representation without additivity.
\newblock {\em Proc. Amer. Math. Soc.}, 97(2):255--261, 1986.

\bibitem{Schmeidler89}
David Schmeidler.
\newblock Subjective probability and expected utility without additivity.
\newblock {\em Econometrica}, 57(3):571--587, 1989.

\bibitem{Segal51}
Irving~E. Segal.
\newblock Equivalences of measure spaces.
\newblock {\em Amer. J. Math.}, 73:275--313, 1951.

\bibitem{Shafer87}
Glenn Shafer.
\newblock Belief functions and possibility measures.
\newblock In {\em Analysis of fuzzy information, {V}ol.\ {I}}, pages 51--84.
  CRC, Boca Raton, FL, 1987.

\bibitem{Shilkret71}
Niel Shilkret.
\newblock Maxitive measure and integration.
\newblock {\em Nederl. Akad. Wetensch. Proc. Ser. A 74 = Indag. Math.},
  33:109--116, 1971.

\bibitem{Stoev05}
Stilian~A. Stoev and Murad~S. Taqqu.
\newblock Extremal stochastic integrals: a parallel between max-stable
  processes and {$\alpha$}-stable processes.
\newblock {\em Extremes}, 8(4):237--266, 2005.

\bibitem{Sugeno74}
Michio Sugeno.
\newblock {\em Theory of fuzzy integrals and its applications}.
\newblock PhD thesis, Tokyo Institute of Technology, Japan, 1974.

\bibitem{Sugeno87}
Michio Sugeno and Toshiaki Murofushi.
\newblock Pseudo-additive measures and integrals.
\newblock {\em J. Math. Anal. Appl.}, 122(1):197--222, 1987.

\bibitem{vanDeVel93}
Marcel L.~J. van~de Vel.
\newblock {\em Theory of convex structures}, volume~50 of {\em North-Holland
  Mathematical Library}.
\newblock North-Holland Publishing Co., Amsterdam, 1993.

\bibitem{Vitali25b}
Giuseppe Vitali.
\newblock On the definition of integral of functions of one variable.
\newblock {\em Riv. Mat. Sci. Econom. Social.}, 20(2):159--168, 1997.
\newblock Translated from the Italian by Massimo Marinacci.

\bibitem{Wang82}
P.-Z. Wang.
\newblock Fuzzy contactability and fuzzy variables.
\newblock {\em Fuzzy Sets and Systems}, 8(1):81--92, 1982.

\bibitem{Wang92}
Zhenyuan Wang and George~J. Klir.
\newblock {\em Fuzzy measure theory}.
\newblock Plenum Press, New York, 1992.

\bibitem{Wang09}
Zhenyuan Wang and George~J. Klir.
\newblock {\em Generalized measure theory}, volume~25 of {\em IFSR
  International Series on Systems Science and Engineering}.
\newblock Springer, New York, 2009.

\bibitem{Wang05}
Zhenyuan Wang, Kwong-Sak Leung, and George~J. Klir.
\newblock Applying fuzzy measures and nonlinear integrals in data mining.
\newblock {\em Fuzzy Sets and Systems}, 156(3):371--380, 2005.

\bibitem{Weber84}
Siegfried Weber.
\newblock {$\perp $}-decomposable measures and integrals for {A}rchimedean
  {$t$}-conorms {$\perp $}.
\newblock {\em J. Math. Anal. Appl.}, 101(1):114--138, 1984.

\bibitem{Weber86}
Siegfried Weber.
\newblock Two integrals and some modified versions---critical remarks.
\newblock {\em Fuzzy Sets and Systems}, 20(1):97--105, 1986.

\bibitem{Yang85a}
Qing~Si Yang.
\newblock The pan-integral on a fuzzy measure space.
\newblock {\em Fuzzy Math.}, 5(3):107--114, 1985.

\bibitem{Zadeh78}
Lotfi~A. Zadeh.
\newblock Fuzzy sets as a basis for a theory of possibility.
\newblock {\em Fuzzy Sets and Systems}, 1(1):3--28, 1978.

\end{thebibliography}

\def\cprime{$'$} 
  \def\ocirc#1{\ifmmode\setbox0=\hbox{$#1$}\dimen0=\ht0 \advance\dimen0
  by1pt\rlap{\hbox to\wd0{\hss\raise\dimen0
  \hbox{\hskip.2em$\scriptscriptstyle\circ$}\hss}}#1\else {\accent"17 #1}\fi}

\appendix

\section{Some properties of $\sigma$-additive measures}\label{messigadd}


The notions of 
$\sigma$-principal or CCC measures were originally introduced for the study of $\sigma$-additive measures. Recall that a $\sigma$-additive measure $m$ defined on a $\sigma$-algebra $\mathrsfs{B}$ is \textit{CCC} (resp.\ \textit{$\sigma$-principal}) if the $\sigma$-maxitive measure $\delta_m$ is. 
Also, following Segal \cite{Segal51}, $m$ is \textit{localizable} if, for all $\sigma$-ideals $\mathrsfs{I}$ of $\mathrsfs{B}$, there exists some $L \in \mathrsfs{B}$ such that 
\begin{enumerate}
	\item $m(S \setminus L) = 0$, for all $S \in \mathrsfs{I}$; 
	\item if there is some $B \in \mathrsfs{B}$ such that $m(S \setminus B) = 0$ for all $S \in \mathrsfs{I}$, then $m(L \setminus B) = 0$. 
\end{enumerate}


The next theorem establishes a link between these notions for $\sigma$-additive measures. It enlightens the fact that being finite is a very strong condition for a $\sigma$-additive measure (while it is of little consequence for a $\sigma$-maxitive measure). 

\begin{theorem}\label{implications}
Let $(E,\mathrsfs{B})$ is a measurable space and $m$ be a $\sigma$-additive measure on $\mathrsfs{B}$. Consider the following assertions:
\begin{enumerate}
	\item\label{implications1} $m$ is finite, 
	\item\label{implications2} $m$ is $\sigma$-finite, 
	\item\label{implications3} $m$ is $\sigma$-principal, 
	\item\label{implications4} $m$ is CCC, 
	\item\label{implications5} $m$ is localizable.  
\end{enumerate}
Then \eqref{implications1} $\Rightarrow$ \eqref{implications2} $\Rightarrow$ \eqref{implications3} $\Rightarrow$ \eqref{implications4} $\Rightarrow$ \eqref{implications5}. Moreover, \eqref{implications4} $\Rightarrow$ \eqref{implications3} under Zorn's lemma. 
\end{theorem}

\begin{proof}[Sketch of the Proof]
Assume that $m$ is finite, and let us show that $m$ is $\sigma$-principal. Let $\mathrsfs{I}$ be a $\sigma$-ideal of $\mathrsfs{B}$. Let $a = \sup \{ m(S) : S \in \mathrsfs{I} \}$. We can find some sequence $S_n \in \mathrsfs{I}$ such that $m(S_n) \uparrow a$. Defining $L := \cup_n S_n \in \mathrsfs{I}$, we have  $m(L) = a$. If there exists some $S \in\mathrsfs{I}$ such that $m(S \setminus L) > 0$, then $m(S \cup L) > a$ (since $m$ is finite), which contradicts  $S \cup L \in \mathrsfs{I}$. Thus, $m(S \setminus L) = 0$, for all $S \in \mathrsfs{I}$, which gives $\sigma$-principality of $m$. 
The other implications in Theorem~\ref{implications} can be proved along the same lines as for $\sigma$-maxitive measures. 
\end{proof}

\section{Residual semigroups}\label{sec:residual}

An \textit{ordered semigroup} is a semigroup $(S,\odot)$ equipped with a partial order $\leqslant$ compatible with the structure of semigroup, i.e.\ such that $r \leqslant s$ and $r' \leqslant s'$ imply $r \odot r' \leqslant s \odot s'$. 

If $(S,\odot)$ is an ordered semigroup and $r, s \in S$, we say that $r$ is \textit{absolutely continuous with respect to $s$}, written $r \ll_{\odot} s$, if there exists some $t\in S$ such that $r \leqslant t \odot s$. 
We say that $S$ (or $\odot$) is \textit{residual} if for all $r, s \in S$ with $r \ll_{\odot} s$, there is an element of $S$ denoted by $(r/s)_{\odot}$ such that $r \leqslant t \odot s \Leftrightarrow (r/s)_{\odot} \leqslant t$, for all $t \in S$. 
Note that in this situation we have $r \leqslant (r/s)_{\odot} \odot s$. 
A residual semigroup $(S, \odot)$ is \textit{exact} if $r = (r/s)_{\odot} \odot s$ for all $r, s \in S$ with $r \ll_{\odot} s$. 

\begin{examples}
In $\overline{\mathbb{R}}_+$ here is what we have for different choices of semigroup binary operations (recall that $\oplus$ denotes the maximum and $\wedge$ the minimum): 
\begin{itemize}
	\item $r \ll_{\times} s \Leftrightarrow (r = s = 0 \mbox{ or } s\neq 0)$, in which case $(r/s)_{\times} \times s = r$. So $(\overline{\mathbb{R}}_+, \times)$ is an exact residual semigroup. 
	\item $r \ll_{+} s$ always holds, and $(r/s)_{+} = 0 \oplus (r - s)$. So $(\overline{\mathbb{R}}_+, +)$ is a non-exact residual semigroup. 
	\item $r \ll_{\oplus} s$ always holds, and $(r/s)_{\oplus} = 0$ if $r \leqslant s$, $(r/s)_{\oplus} = r$ otherwise. So $(\overline{\mathbb{R}}_+, \oplus)$ is a non-exact residual semigroup. 
	\item $r \ll_{\wedge} s \Leftrightarrow r \leqslant s$, in which case $(r/s)_{\wedge} = r$, so $(\overline{\mathbb{R}}_+, \wedge)$ is an exact residual semigroup.  
\end{itemize}
\end{examples}

\begin{proposition}
Let $(S,\odot)$ be an ordered semigroup. 
If $S$ is residual, then for all nonempty subsets $T$ of $S$ with infimum and all $s \in S$, $\{ t \odot s : t \in T \}$ has an infimum and 
\begin{equation}\label{eq:inf}
\inf_{t \in T} (t \odot s) = (\inf T) \odot s. 
\end{equation}
Conversely, if every non-empty subset of $S$ has an infimum and Equation~\eqref{eq:inf} is satisfied for all nonempty subsets $T$ of $S$ with infimum and all $s \in S$, then $S$ is residual. 
\end{proposition}

\begin{proof}
First assume that $S$ is residual. 
Let $T$ be a nonempty subset of $S$ with infimum, and let $s \in S$. Then $(\inf T) \odot s$ is a lower-bound of the set $A = \{ t \odot s : t \in T \}$. 
Now let $\ell$ be a lower-bound of $A$. 
Since $T$ is non-empty we have $\ell \ll_{\odot} s$. Moreover, $\ell \leqslant t \odot s$ for all $t \in T$, so that $(\ell / s)_{\odot} \leqslant t$ for all $t \in T$. This shows that $(\ell / s)_{\odot} \leqslant \inf T$, i.e.\ that $\ell \leqslant (\inf T) \odot s$. So $(\inf T) \odot s$ is the greatest lower bound of $A$, i.e.\ its infimum, and we have proved Equation~\eqref{eq:inf}. 

Conversely, assume that every non-empty subset of $S$ has an infimum and that Equation~\eqref{eq:inf} is satisfied, and let $r, s \in S$ such that $r \ll_{\odot} s$. Define $(r / s)_{\odot} = \inf T$, where $T$ is the nonempty set $\{ t \in S : r \leqslant t \odot s \}$. 
Thanks to Equation~\eqref{eq:inf}, the equivalence $r \leqslant t \odot s \Leftrightarrow (r / s)_{\odot} \leqslant t$, for all $t \in S$, is now obvious. 
So $S$ is residual. 
\end{proof}

\end{document}